\documentclass[a4paper, 11pt]{article}
\usepackage{amsmath,amssymb,esint,amscd,xspace,fancyhdr,color,authblk,srcltx,fontenc,bbm}
\usepackage{hyperref}
\usepackage{cite}

\newtheorem{theorem}{Theorem}[section]

\newtheorem{definition}[theorem]{Definition}
\newtheorem{lemma}[theorem]{Lemma}
\newtheorem{proposition}[theorem]{Proposition}
\newtheorem{remark}[theorem]{Remark}

\newtheorem{taggedtheoremx}{Theorem}
\newenvironment{taggedtheorem}[1]
 {\renewcommand\thetaggedtheoremx{#1}\taggedtheoremx}
 {\endtaggedtheoremx}
 
\newenvironment{proof}[1][Proof]{\textbf{#1.} }{\hfill\rule{0.5em}{0.5em}}
{\catcode`\@=11\global\let\AddToReset=\@addtoreset
\AddToReset{equation}{section}

\AddToReset{theorem}{section}

\title{Gradient estimates via Riesz potentials and fractional maximal operators for quasilinear elliptic equations with applications}

\author{Minh-Phuong Tran\footnote{Corresponding author.} \thanks{Applied Analysis Research Group, Faculty of Mathematics and Statistics, Ton Duc Thang University, Ho Chi Minh city, Vietnam; \texttt{tranminhphuong@tdtu.edu.vn}}, Thanh-Nhan Nguyen\thanks{Department of Mathematics, Ho Chi Minh City University of Education, Ho Chi Minh city, Vietnam; \texttt{nhannt@hcmue.edu.vn}}}

\date{\today}

\begin{document}
 
\maketitle
\begin{abstract}
In this paper, the aim of our work is to establish global weighted gradient estimates via fractional maximal functions and the point-wise regularity estimates of Dirichlet problem for divergence elliptic equations of the type
\begin{align*}
\mathrm{div}(A(x,\nabla u)) =  \mathrm{div}(f) \ \text{in} \  \Omega, \mbox{ and } \ u =  g \ \text{on} \  \partial \Omega,
\end{align*}
that related to Riesz potentials. Here, in our setting, $\Omega \subset \mathbb{R}^n$, $n \ge 2$ is a bounded Reifenberg flat domain (that its boundary is sufficiently flat in sense of Reifenberg) and the small-BMO condition (small bounded mean oscillations) is assumed on the nonlinearity $A$. Further, the emphasis of the paper is the existence of weak solution to a class of quasilinear elliptic equations containing Riesz potential of the gradient term, as an application of the global point-wise bound. And regarding this study, we also analyze the necessary and sufficient conditions that guarantee the existence of solution to such nonlinear elliptic problems. 
\noindent 

\medskip


\end{abstract}   
                  
\tableofcontents
\section{Introduction and the statement of main results}
\label{sec:intro}
In this paper, we are concerned with a class of nonhomogeneous quasilinear elliptic problem
\begin{align}
\tag{P}
\label{eq:diveq}
\begin{cases}
\mathrm{div}(A(x,\nabla u)) &= \ \mathrm{div}(f) \quad \text{in} \ \ \Omega, \\
\hspace{1.2cm} u &= \ g \quad \qquad \text{on} \ \ \partial \Omega,
\end{cases}
\end{align}
where $\Omega \subset \mathbb{R}^n$ is an open subset, $n \ge 2$; the function $A: \Omega \times \mathbb{R}^n \to \mathbb{R}^n$ is a Carath\'edory vector field satisfying
\begin{align*}
\begin{cases}
\left| A(x,\xi) \right| \le \Lambda_1 |\xi|^{p-1},\\
\langle A(x,\xi)-A(x,\eta), \xi - \eta \rangle \ge \Lambda_2 \left( |\xi|^2 + |\eta|^2 \right)^{\frac{p-2}{2}}|\xi - \eta|^2,
\end{cases}
\end{align*}
for every $(\xi,\eta) \in \mathbb{R}^n \times \mathbb{R}^n \setminus \{(0,0)\}$ and a.e. $x \in \mathbb{R}^n$, $\Lambda_1$ and $\Lambda_2$ are positive constants. The left-hand side operator $\mathrm{div}(A(x,\nabla u))$ is considered as the more general form of the $p$-Laplacian $\Delta_pu :=-\mathrm{div}(|\nabla u|^{p-2}\nabla u)$. This operator and its properties will be clarified in Section \ref{sec:pre}. Additionally, one has the boundary data $g \in W^{1,p}(\Omega;\mathbb{R})$ for some $p>1$ and $f \in L^{p'}$, $p'$ is the exponent conjugate to $p$. Problems of type \eqref{eq:diveq} have been devoted to study in the last several decades, specifically the elliptic equations involving $p$-Laplacian. As well-known, problems involving $p$-Laplace operator typically arise in contexts of physical phenomena and have a wide range of applications, such as nonlinear elasticity, reaction-diffusion problem, and the study of non-Newtonian fluid, etc. \cite{Ruzicka, Zhikov,Lindq}.

Many recent studies have been focused on the regularity theory for nonlinear elliptic equations, where the nonlinearities are formulated around the $p$-Laplacian. The question of optimal regularity for elliptic equations divergence form has attracted a lot of attention of mathematicians for many years. Regularity of solutions to elliptic problems with homogeneous Dirichlet problem (zero Dirichlet boundary data) is a classical topic. Various mathematical techniques have been developed to obtained local integrability/gradient estimates for such problems \cite{CP1998, Mi3}. And later, there have been various results pertaining to up-to-boundary (or global) regularity estimates when suitable assumptions are given on $\partial \Omega$, the vector field $A$ and functional data $f,g$. The reader can find a plenty of materials related to this topic by E. DiBenedetto in \cite{DiBenedetto1983, Lieberman1984, Lieberman1986, Lieberman1988,DiBenedetto1993}; by T. Iwaniec in \cite{Iwaniec}; S. Byun \textit{et al.} in \cite{SSB4, SSB2, BW1_1, BW2} and further generalization to this type of homogeneous equation are the subjects of \cite{CM2014, Phuc2015, CoMi2016, BCDKS, BW1, KZ,Mi1} and their related references.  Later, more extensions of regularity to the non-homogeneous quasilinear elliptic equations of the form \eqref{eq:diveq} in Lorentz-type spaces, as well as in Morrey-type spaces were discussed and addressed in many papers, such as \cite{Truyen2018,Tuoc2018,MPTN2019}. Specifically, in \cite{MPTN2019}, we established the global estimates for gradients of solutions to problem \eqref{eq:diveq} via the use of fractional maximal operators $\mathbf{M}_\alpha$. According to the past papers by Kuusi \textit{et al.} \cite{KMin2013, KMin2014}, authors mentioned the important role of Hardy-Littlewood and fractional maximal operators in studying the theory of partial differential equations,  differentiability properties of functions, singular integrals, etc. Fractional maximal function has a relation to Riesz potential (fractional integral operator) due to the following point-wise inequality:
\begin{align*}
\mathbf{M}_\alpha u(x) \le C \mathbf{I}_\alpha u(x), \quad \text{for every} \ x \in \mathbb{R}^n.
\end{align*}
Moreover, as shown in \cite[Theorem 1]{Muck74}, the converse inequality holds in its integral form as below:
\begin{align*}
\int_{\mathbb{R}^n}{(\mathbf{I}_\alpha u)^q \omega dx} \le C \int_{\mathbb{R}^n}{(\mathbf{M}_\alpha u)^q\omega dx}, \quad \text{for} \ q>0 \ \text{and} \ \omega\in \mathcal{A}_\infty.
\end{align*}
Therefore, gradient estimates for solutions to elliptic problems via fractional maximal functions not only provide information of size and oscillations of solutions and their derivatives, but also allow to bound fractional derivatives of $u$: $\partial^\alpha u$, for $0 \le \alpha <2$. Readers may consult \cite{KMin2014} and references given there to explore more.

There are two main studies in this paper. The first one is devoted to the study of global gradient estimates for solutions to \eqref{eq:diveq} in terms of fractional maximal operators $\mathbf{M}_\alpha$. A point worth emphasizing here is that, for better results than regularity treated in our previous paper \cite{MPTN2019}, both interior and boundary results will be obtained under an additional structural assumption on $A$ (that satisfies the small-BMO condition) and geometric assumption on $\partial \Omega$ (Reifenberg flat domain). For the second result in this paper, we are interested in finding point-wise gradient estimates for solutions to \eqref{eq:diveq} in terms of both Riesz potentials and fractional maximal functions as mentioned above. It is known that in recent papers \cite{KMin2013,KMin2014,Mi4}, Kuusi and Mingione firstly proved the point-wise gradient estimates for solutions to elliptic equations with measure data using linear and nonlinear potentials. The approach we take lies close in spirit to such ideas, and a result of point-wise estimate by Riesz potential for gradient of solutions will be established here. It is worthwhile to note that we pay especial attention to gradient estimates preserved under the fractional maximal functions $\mathbf{M}_\alpha$. Besides, this work also deals with the study of the existence result for problem when the gradient source term is driven under a certain Riesz potential. Among the recent works that studied the existence of solutions with gradient source term as in \cite{BNV2020, 55QH4, Phuc2010, Phuc2008, Phuc2014} and so on, results obtained in this paper can be a contribution towards the understanding of regularity theory and applications to many types of nonlinear problems.

Let us now give precise statements of our main results, via some main theorems presented as below. The following theorem  \ref{theo:regularity} establishes the estimates on gradients of solutions in terms of fractional maximal functions. Further, our results deal with data in the setting of Lorentz spaces with Muckenhoupt weights. 

\begin{taggedtheorem}{A}\label{theo:regularity}
Let $p>1$,  $f \in L^{\frac{p}{p-1}}(\Omega; \mathbb{R}^n)$, $g \in W^{1,p}(\Omega;\mathbb{R})$ and $u$ be a weak solution to \eqref{eq:diveq}. For any $\alpha \in [0,n)$, $\omega \in \mathcal{A}_\infty$, $0<q<\infty$, $0<s\le \infty$, there exists a constant $\delta = \delta(n,p,[\omega]_{\mathcal{A}_\infty})$ such that if $\Omega$ is a $(\delta,R_0)$-Reifenberg flat domain satisfying $[A]_{R_0} \le \delta$ for some $R_0>0$, then
\begin{align}\label{eq:reg}
\|\mathbf{M}\mathbf{M}_{\alpha}(|\nabla u|^p)\|_{L^{q,s}_{\omega}(\Omega)} \le C \|\mathbf{M}_{\alpha}(|f|^{\frac{p}{p-1}} + |\nabla g|^p)\|_{L^{q,s}_{\omega}(\Omega)}.
\end{align}
The constant $C$ here depends only on $n,p,\alpha,q,s,[\omega]_{\mathcal{A}_\infty},\mathrm{diam}(\Omega)/R_0$.
\end{taggedtheorem}
Throughout the paper, the denotation $\mathrm{diam}(\Omega)$ is the diameter of a set $\Omega$ defined as:
\begin{align*}
\mathrm{diam}(\Omega) = \sup\{d(x,y) \  : \ x,y \in \Omega\},
\end{align*}
and $\mathcal{A}_\infty$ the Muckenhoupt weights will be described in Section \ref{sec:pre} later. Moreover, here and hereforth, for simplicity, the set $\{x \in \Omega: |g(x)| > \Lambda\}$ is denoted by $\{|g|>\Lambda\}$. It also emphasizes here that in order to obtain such gradient bound, method of using the `good-$\lambda$' technique is in use and adjusted to the problem \eqref{eq:diveq} with additional assumptions on $A$ and $\partial\Omega$ as aforementioned. This method was first proposed in \cite{AM2007} and later modified, improved in various remarkable papers \cite{55QH4,MPT2018,MPT2019,MPTN2019, BW1, SSB2,Phuc2015,BCDKS}. Recently, it becomes a promising technique adopted in regularity estimates of nonlinear elliptic equations among many other approaches developed during the last years. In Section \ref{sec:main} we will state and prove the good-$\lambda$ theorem that associated with our problem \eqref{eq:diveq}.

The above theorem yields the following point-wise gradient estimates of solutions in term of Riesz potential $\mathbf{I}_\beta$.

\begin{taggedtheorem}{B}\label{theo:I-alpha}
Let $p>1$,  $f \in L^{\frac{p}{p-1}}(\Omega; \mathbb{R}^n)$, $g \in W^{1,p}(\Omega;\mathbb{R})$ and $u$ be a weak solution to \eqref{eq:diveq} in a $(\delta,R_0)$-Reifenberg flat domain $\Omega$ for sufficiently small $\delta>0$, with $[A]_{R_0} \le \delta$ for some $R_0>0$. Then for any $\alpha \in [0,n)$, $\beta \in (0,n)$, $0<q<\infty$, the following point-wise estimate
\begin{align}\label{eq:I-alpha}
\mathbf{I}_{\beta}\left(|\mathbf{M}_{\alpha}(|\nabla u|^p)|^q \chi_{\Omega}\right)(x) \le C \mathbf{I}_{\beta}\left(|\mathbf{M}_{\alpha}(|f|^{\frac{p}{p-1}} + |\nabla g|^p)|^q \chi_{\Omega} \right)(x)
\end{align}
holds for almost everywhere $x \in \mathbb{R}^n$.
\end{taggedtheorem}

Furthermore, as an application of such point-wise gradient bound from Theorem \ref{theo:I-alpha}, in this paper, we are concerned with the existence of solutions to equations of the type:
\begin{align}\tag{Q}
\label{eq:existI}
\begin{cases} -\mathrm{div}(A(x,\nabla u)) & = \ \mathbf{I}_{\beta}(|\nabla u|^p)^q + \mathrm{div}(f), \ \mbox{ in } \ \Omega, \\ \hspace{1.5cm} u & = \  g, \ \mbox{ on } \partial \Omega,\end{cases}
\end{align}
and our proof rests on the well-known Riesz potentials and the Riesz capacity condition (see \cite{Phuc2008, BNV2018, BNV2020, V2,55QH4} for related results). We now respectively state two theorems (Theorem \ref{theo:existence_Ialpha} and Theorem \ref{theoE} below) concerning the necessary and sufficient conditions which address the existence of solutions to problem \eqref{eq:existI}.
\begin{taggedtheorem}{C}\label{theo:existence_Ialpha}
Let $\alpha, \, \beta \in (0,n)$, $p>1$, $\frac{p-1}{p}<q<\frac{n}{n-\beta}$, $f \in L^{\frac{p}{p-1}}(\Omega; \mathbb{R}^n)$ and $g \in W^{1,p}(\Omega;\mathbb{R})$. There exist some positive constants $\delta$, $\varepsilon$ such that if $\Omega$ is a $(\delta,R_0)$-Reifenberg flat domain satisfying $[A]_{R_0} \le \delta$ for some $R_0>0$ and the functional data $|\mathcal{F}|^p: = |f|^{\frac{p}{p-1}} + |\nabla g|^p$ satisfies the following inequality
\begin{align}\label{eq:ReiszCapa_cond}
\mu(K) \le \varepsilon \ \mathrm{Cap}_{\mathbf{I}_{\beta + \frac{1}{q}},\frac{pq}{pq-p+1}}(K),
\end{align}
for any compact set $K \subset \mathbb{R}^n$ with $d\mu  = |\mathcal{F}(x)|^p dx$, then the equation~\eqref{eq:existI} admits at least a solution $u \in W^{1,p}(\Omega)$ and there holds
\begin{align}\label{eq:theo-D}
\mathbf{I}_{\alpha}(|\nabla u|^p)(x) \le \Lambda \mathbf{I}_{\alpha} (|\mathcal{F}|^p)(x), \quad  \mbox{ for a.e. } x \in \mathbb{R}^n,
\end{align}
for a constant $\Lambda>0$.
\end{taggedtheorem}

As in \cite{AH, V2}, the condition \eqref{eq:ReiszCapa_cond} is known as Riesz capacity condition, where the $(\alpha,p)$-capacity ${\mathrm{Cap}}_{\mathbf{I}_\alpha,p}(K)$ corresponds to the Sobolev spaces $W^{\alpha,p}(\mathbb{R}^n)$ of a compact set $K$ is defined by
\begin{align*}
{\mathrm{Cap}}_{\mathbf{I}_\alpha,p}(K) = \inf \left\{\int_{\mathbb{R}^n}{|\phi(x)|^pdx}: \ \phi \in L_+^p(\mathbb{R}^n), \, \mathbf{I}_\alpha[\phi] \ge \chi_K \right\}.
\end{align*}

It is worth mentioning that \eqref{eq:ReiszCapa_cond} is a sufficient condition but not necessary condition for the existence of solutions to \eqref{eq:existI}. Theorem \ref{theoE} below gives \eqref{eq:ReiszCapa_cond-b} as the necessary condition to guarantee the existence. Moreover, it is clear that when $\frac{1}{q} + \frac{\beta}{p} = 1$, the condition makes it both necessary and sufficient for the validity of existence.

\begin{taggedtheorem}{D}\label{theoE}
Let $\beta \in (0,n)$, $p>1$, $q > \max\big\{\frac{p-1}{p},\frac{\beta}{\beta+1}\big\}$ and $\mu \in \mathcal{M}^+(\Omega)$. There exists a positive constant $\delta$ such that if $\Omega$ is a $(\delta,R_0)$-Reifenberg flat domain satisfying $[A]_{R_0} \le \delta$ for some $R_0>0$ and the following equation
\begin{align}\label{eq:existI-b}
\begin{cases} -\mathrm{div}(A(x,\nabla u)) & = \ \mathbf{I}_{\beta}(|\nabla u|^p)^q + \mu, \ \mbox{ in } \ \Omega, \\ \hspace{1.5cm} u & = \  0, \ \mbox{ on } \partial \Omega,\end{cases}
\end{align}
admits a renormalized solution $u$, then one can find a constant $C$ such that
\begin{align}\label{eq:ReiszCapa_cond-b}
\mu(K) \le C \ \mathrm{Cap}_{\mathbf{I}_{\beta + 1- \frac{\beta}{q}},\frac{pq}{pq-p+1}}(K),
\end{align}
for any compact set $K \subset \mathbb{R}^n$. 
\end{taggedtheorem}

The rest of our paper is organized as follows. We start in Section \ref{sec:pre} by introducing and collecting some standard notations, assumptions in which our problem is formulated. Section \ref{sec:comparison} is dedicated to the interior and boundary comparison estimates on the solutions, some preparatory lemmas in this section also present a basic idea that allows us to prove results. In Section \ref{sec:main} we drive the so-called ``good-$\lambda$'' technique to obtain the gradient estimates for the fractional maximal operators and the point-wise gradient bounds for solutions in terms of Riesz potentials, the proofs of gradient estimate theorems are also given in this section. The last section \ref{sec:app} is then devoted to proving Theorem~\ref{theo:existence_Ialpha} and Theorem~\ref{theoE}, an application that may interact with many mathematical or physical equations in many fields of science.

\section{Notations and Preliminaries}
\label{sec:pre}

This section consists of some necessary preliminaries in which our problem is formulated, and we also recall some well-known notations, fundamentals and results for later use herein.  

\subsection{Notation and definitions}
\label{subsec:def}

Throughout the study, we recall that the denotation $B_r(x)$ stands for an open ball in $\mathbb{R}^n$ with radius $r$ and centered at $x$, that is the set $B_r(x) = \{y\in \mathbb{R}^n: |y-x|<r\}$. For convenience of the reader, we use $|B|$ stands for the $n$-dimensional Lebesgue measure of a set $B \subset \mathbb{R}^n$. And in what follows, let us denote by $\displaystyle{\fint_{B_r(x)}{f(y)dy}}$ indicates the integral average (mean value) of $f$ in the variable $y$ over the ball $B_r(x)$, i.e.
\begin{align*}
\fint_{B_r(x)}{f(y)dy} = \frac{1}{|B_r(x)|}\int_{B_r(x)}{f(y)dy}.
\end{align*}

In the present paper, the considered domain $\Omega \subset \mathbb{R}^n$ is assumed to be a bounded $(\delta,R_0)$-Reifenberg flat domain, whose definition is stated as follows. 
\begin{definition}[$(\delta,R_0)$-Reifenberg flat domain]\label{def:Reif} 
Let $\delta \in (0,1)$ and $R_0>0$. We say that $\Omega$ is a $(\delta,R_0)$-Reifenberg flat domain if for every $x\in \partial \Omega$ and $0< r < R_0(1-\delta)$, there exists a coordinates system $\{y_1,y_2,...,y_n\}$ such that in this coordinate system $x = - r\delta/(1-\delta)y_n$ and
\begin{align*}
B_r(0) \cap \{y_n > 0\} \subset B_r(0) \cap \Omega \subset B_r(0) \cap \{y_n > -2 r\delta/(1-\delta) \},
\end{align*}
where we denote the set $\{y = (y_1, y_2, ..., y_n): \ y_n > c\}$ by $\{y_n > c\}$.
\end{definition}

Reifenberg flat domain has its boundary with nice feature, that can be approximated by hyperplanes. This type of domain was first described by Reifenberg in \cite{Reifenberg} when attacking the Plateau problem. In this celebrated paper, he proved that the $(\delta,R_0)$-Reifenberg flat domain for $\delta>0$ small enough represents locally a topological disc. It will be better to understand that  Lipschitz domains with small Lipschitz constant, $C^1$–domains, and von Koch snowflakes or some certain quasi-balls are domains which are `flat' in this sense. Here, we refer the reader to \cite{Reifenberg, Toro1997} for additional details.

Furthermore, in the setting of our problem, the nonlinear operator $A: \Omega \times \mathbb{R}^n \rightarrow \mathbb{R}$ is a Carath\'eodory vector valued function which satisfies the following growth and monotonicity conditions: for some $1<p\le n$ there exist two positive constants $\Lambda_1$ and $\Lambda_2$ such that 
\begin{align}\label{cond:A1}
\left| A(x,\xi) \right| &\le \Lambda_1 |\xi|^{p-1},
\end{align}
and
\begin{align}\label{cond:A2}
\langle A(x,\xi)-A(x,\eta), \xi - \eta \rangle &\ge \Lambda_2 \left( |\xi|^2 + |\eta|^2 \right)^{\frac{p-2}{2}}|\xi - \eta|^2
\end{align}
holds for almost every $x$ in $\Omega$ and every $\xi$, $\eta \in \mathbb{R}^n \setminus \{0\}$. Moreover, in our regularity proofs, the operator $A$ is also assumed to satisfy a sufficiently small bounded mean oscillation (small-BMO) condition, that is described as below.

\begin{definition}[small-BMO condition]\label{def:BMOcond}
Let $\delta >0$ and $R_0>0$. We say that the nonlinearity $A$ satisfies a $(\delta,R_0)$-BMO condition (or shortly said BMO-condition) if
\begin{align*}
[A]_{R_0} = \sup_{y \in \mathbb{R}^n, \ 0<r\le R_0} \left(\fint_{B_r(y)} \sup_{\xi \in \mathbb{R}^n \setminus \{0\}} \frac{|A(x,\xi) - \overline{A}_{B_r(y)}(\xi)|}{|\xi|^{p-1}}dx\right)  \le \delta,
\end{align*}
where $\overline{A}_{B_r(y)}(\xi)$ denotes the integral average of $A(\cdot,\xi)$ over the ball $B_r(y)$.
\end{definition}

Our work is also related to the class of Muckenhoupt's weights $\mathcal{A}_p$. This concept first appeared by Muckenhoupt in \cite{Muck72} and since then, numerous norm inequalities and boundedness of relevant operators have been established for the $\mathcal{A}_p$ classes in various research approaches. The Muckenhoupt classes of weighted functions are closely connected with the boundedness of Hardy-Littlewood maximal functions. Let us recall the definition of the Muckenhoupt weights and derive some of their properties for later use. Here and subsequently, by a weight $\omega$, we mean that $\omega$ is a non-negative measurable and locally integrable function on $\mathbb{R}^n$. For any measurable set $E\subset \mathbb{R}^{n}$ and the weight $\omega$, we denote 
$$\omega(E):=\int_{E}\omega(x)dx.$$ 

\begin{definition}[Muckenhoupt weights]\label{def:Muck}
For $1\le p \le \infty$, we say that a weight $\omega \in L^1_{\mathrm{loc}}(\mathbb{R}^n)$ belongs to the Muckenhoupt class $\mathcal{A}_p$ if there holds
\begin{align*}
[\omega]_{\mathcal{A}_p} = \sup_{B_r(x) \subset \mathbb{R}^n} \left(\fint_{B_r(x)} \omega(y) dy\right)\left(\fint_{B_r(x)}\omega(y)^{-\frac{1}{p-1}}dy\right)^{p-1}<\infty, \mbox{ when } 1<p<\infty,
\end{align*}
\begin{align*}
[\omega]_{\mathcal{A}_1} = \sup_{B_r(x) \subset \mathbb{R}^n} \left(\fint_{B_r(x)} \omega(y) dy\right) \sup_{y \in B_r(x)} \frac{1}{\omega(y)} <\infty, \mbox{ when } p=1,
\end{align*}
and there are two positive constants $C$ and $\nu$ such that
\begin{align*}
\omega(E) \le C \left(\frac{|E|}{|B|}\right)^\nu \omega(B), \mbox{ when } p=\infty,
\end{align*}
for all ball $B=B_r(x)$ in $\mathbb{R}^n$ and all measurable subset $E$ of $B$. In this case, we denote $[\omega]_{\mathcal{A}_\infty} = (C,\nu)$.
\end{definition}

\begin{remark}\label{rem:muck}
In Definition~\ref{def:Muck}, the number $[\omega]_{\mathcal{A}_p}$  is called the ${\mathcal{A}_p}$ constant of $\omega$ and it is well known that $\mathcal{A}_1 \subset \mathcal{A}_p \subset \mathcal{A}_\infty$ for all $1 \le p \le \infty$. Moreover, the Muckenhoupt class $\mathcal{A}_{\infty}$ is given by: $$\mathcal{A}_\infty = \bigcup_{p<\infty}\mathcal{A}_p.$$ 
\end{remark}

In this paper, the study will be made in the setting of weighted Lorentz spaces, defined as below. And for literature that concerning these spaces, the reader refers to \cite{Lorentz1950,Lorentz1951,Gene98} and textbooks \cite{55Gra,Stein} for detailed information.
 
\begin{definition}[Weighted Lorentz space]
Let $0<q<\infty$, $0<s\le \infty$ and the Muckenhoupt weight $\omega \in \mathcal{A}_{\infty}$. We define the weighted Lorentz space $L^{q,s}_{\omega}(\Omega)$ by the set of all Lebesgue measurable functions $h$ on $\Omega$ such that $\|h\|_{L^{q,s}_{\omega}(\Omega)}<+\infty$, where
\begin{align}\label{eq:w-lorentz-a}
\|h\|_{L^{q,s}_{\omega}(\Omega)} = \begin{cases}\left[ q \int_0^\infty{ \lambda^s\omega \left( \{x \in \Omega: |h(x)|>\lambda\} \right)^{\frac{s}{q}} \frac{d\lambda}{\lambda}} \right]^{\frac{1}{s}}, & \ \mbox{ if } s < \infty, \\
\sup_{\lambda>0}{\lambda \omega\left(\{x \in \Omega:|h(x)|>\lambda\}\right)^{\frac{1}{q}}}, & \ \mbox{ if } s = \infty.\end{cases}
\end{align}
\end{definition}
In this way, when $\omega=1$, the weighted Lorentz space $L^{q,s}_{\omega}(\Omega)$ becomes the unweighted (classical) Lorentz space $L^{q,s}(\Omega)$. Moreover, in the case of weighted Lorentz spaces, when $q = s$, $L^{q,s}_{\omega}(\Omega)$ coincides the weighted Lebesgue space $L^{q}_{\omega}(\Omega)$ which is defined by the set of all measurable functions $h$ such that
\begin{align*}
\|h\|_{L^{q}_{\omega}(\Omega)} := \left(\int_{\Omega} |h(x)|^q\omega(x) dx\right)^{\frac{1}{q}} < + \infty.
\end{align*}

\begin{definition}[Riesz potential]
Let $n \ge 2$ and the Riesz potential $\mathbf{I}_\beta$ of order $\beta \in (0,n)$  of a measurable function $h \in L^1_{\mathrm{loc}}(\mathbb{R}^n;\mathbb{R}^+)$ is defined as the convolution:
\begin{align}
\label{def:Riesz}
\mathbf{I}_\beta (h)(x) \equiv (\mathbf{I}_\beta * h)(x) = \int_{\mathbb{R}^n}{\frac{h(y)}{|x-y|^{n-\beta}}dy}, \quad x \in \mathbb{R}^n.
\end{align}
\end{definition}

\begin{definition}[Wolff potential]
Let $\alpha \in (0,n)$ and $1 < \beta <\frac{n}{\alpha}$. The Wolff potential $\mathbf{W}_{\alpha,\beta}(\nu)$ of a non-negative Borel measure $\nu$ is defined as the convolution:
\begin{align*}
\mathbf{W}_{\alpha,\beta} (\nu)(x) = \int_0^{\infty} \left(\frac{\nu(B_r(x))}{r^{n-\alpha\beta}}\right)^{\frac{1}{\beta -1}} \frac{dr}{r}, \quad x \in \mathbb{R}^n.
\end{align*}
\end{definition}
We write $\mathbf{W}_{\alpha,\beta} (h)$ instead of $\mathbf{W}_{\alpha,\beta} (\nu)$ if $d\nu = hdx$, where $h \in L^1_{\mathrm{loc}}(\mathbb{R}^n;\mathbb{R}^+)$. We also remark that $\mathbf{I}_{\alpha}(\nu) = \mathbf{W}_{\frac{\alpha}{2},2}(\nu)$.

\subsection{Fractional Maximal functions}\label{sec:MMalpha}

We now recall the definition of fractional maximal function that regarding to \cite{K1997, KS2003}. Let $0 \le \alpha \le n$, the fractional  maximal function $\mathbf{M}_\alpha$ of a locally integrable function $h \in L^1_{\mathrm{loc}}(\mathbb{R}^n;\mathbb{R})$ is defined by:
\begin{align*}
\mathbf{M}_\alpha h(x) = \sup_{\rho>0}{\rho^\alpha \fint_{B_\rho(x)}{|h(y)|dy}}, \qquad x \in \mathbb{R}^n.
\end{align*}
It is worth to remark that for the case $\alpha=0$, the fractional  maximal function $\mathbf{M}_\alpha$ becomes the Hardy-Littlewood maximal function $\mathbf{M}$. The standard and classical properties of the maximal function $\mathbf{M}$ can be found in many places, see for instance \cite{Gra97,55Gra}. Here we recall some well-known properties of maximal and fractional maximal operators, that will be shown in some following lemmas. The reader is referred to \cite{55Gra} for details.
\begin{lemma} \label{lem:boundM}
The maximal function $\mathbf{M}$ is bounded from $L^q(\mathbb{R}^n)$ to $L^{q,\infty}(\mathbb{R}^n)$, for $q \ge 1$, i.e., there exists a positive constant $C$ such that
\begin{align*}
\left|\{x \in \mathbb{R}^n: \ \mathbf{M}h(x)>\lambda\}\right| \le \frac{C}{\lambda^q}\int_{\mathbb{R}^n}{|h(x)|^q dx}, 
\end{align*}
for all $\lambda>0$ and  $h \in L^q(\mathbb{R}^n)$.
\end{lemma}
\begin{lemma}\label{lem:boundMlorentz}
Let $q>1$ and $0 < s \le \infty$, there exists a positive constant $C$ such that
\begin{align*}
\|\mathbf{M}h\|_{L^{q,s}(\Omega)} \le C \|h\|_{L^{q,s}(\Omega)},
\end{align*}
for all $h \in L^{q,s}(\Omega)$.
\end{lemma}
Moreover, a very important property of fractional maximal function was also obtained from the boundedness property of maximal function. The detail proof of this result can be found in~\cite{MPT2019}.
\begin{lemma}\label{lem:Malpha}
Let $0\le \alpha<n$ and for any locally integrable function $h \in L^1_{\mathrm{loc}}(\mathbb{R}^n)$ there holds
\begin{equation*}
{\left|\left\{x \in \mathbb{R}^n: \ \mathbf{M}_\alpha h(x) >\lambda\right\}\right|} \leq  C\left(\frac{1}{\lambda}\int_{\mathbb{R}^n}|h(x)|dx\right)^{\frac{n}{n-\alpha}},
\end{equation*}
for all $\lambda>0$.
\end{lemma}
Due to the importance of fractional maximal operators, recently in~\cite{MPT2019}, we define an additional cut-off maximal function ${\mathbf{M}}^r_{\alpha}$ of a locally integrable function $h$ that corresponding to ${\mathbf{M}}_{\alpha}$ as follows: for $0\le \alpha \le n$ and $r>0$, 
\begin{align*}
{\mathbf{M}}^r_{\alpha}h(x)  = \sup_{0<\rho<r} \rho^{\alpha} \fint_{B_\rho(x)}h(y)dy, \qquad x \in \mathbb{R}^n.
\end{align*}
In the proof-of-work of the same paper~\cite{MPT2019}, we are concerned with an interesting property of the cut-off maximal function ${\mathbf{M}}^r_{\alpha}$, which we state in lemma below. This leads us to the key tools to achieve results in the sequel. We address the reader to~\cite{MPT2019} for the proof of this lemma.
\begin{lemma}\label{lem:MrMr}
Let $0\le \alpha <n$ and $r>0$. There exists a constant $C=C(n,\alpha)>0$ such that
\begin{align}\label{eq:MrMr}
{\mathbf{M}}^r {\mathbf{M}}^r_{\alpha} h(x) \le C {\mathbf{M}}^{2r}_{\alpha}h(x), \mbox{ for all } x \in \mathbb{R}^n,
\end{align}
for any $h \in L^1_{\mathrm{loc}}(\mathbb{R}^n)$.
\end{lemma}

\section{Comparison results}
\label{sec:comparison}

In this section, we present some local interior and boundary comparison estimates for weak solution $u$ of~\eqref{eq:diveq} that are essential to our development later. It is also remarkable that in some statements below and in what follows, we shall adopt the denotation $C$ for a suitable positive constant that is not necessary the same from line to line in each occurrence.

First of all, we can exploit the following integral estimate on gradient of solution $u$ to~\eqref{eq:diveq}, with respect to initial data $f$ and $g$.

\begin{proposition}\label{prop1}
Let $g \in W^{1,p}(\Omega), \ f \in L^{\frac{p}{p-1}}(\Omega)$ and $u$ be a weak solution of~\eqref{eq:diveq}. There exists a positive constant $C = C(n,p,\Lambda_1,\Lambda_2)$ such that
\begin{equation}\label{eq:prop1}
\int_{\Omega} |\nabla u|^p dx \le C \int_{\Omega} \left(|f|^{\frac{p}{p-1}} + |\nabla g|^p \right)dx.
\end{equation}
\end{proposition}
\begin{proof}
By using $u - g$ as a test function of equation~\eqref{eq:diveq}, we obtain
\begin{equation*}
\int_{\Omega} \langle A(x,\nabla u), \nabla u \rangle dx = \int_{\Omega} \langle A(x,\nabla u), \nabla g \rangle dx + \int_{\Omega} \langle f, \nabla (u - g)\rangle dx.
\end{equation*}
Taking into account both conditions of operator $A$ in~\eqref{cond:A1} and~\eqref{cond:A2}, it deduces that
\begin{equation*}
\int_{\Omega} |\nabla u|^p dx \le C \left( \int_{\Omega} |\nabla u|^{p-1} |\nabla g| dx + \int_{\Omega} |f| |\nabla u| dx + \int_{\Omega} |f| |\nabla g| dx \right).
\end{equation*}
Finally, we may easily obtain~\eqref{eq:prop1} by combining the H{\"o}lder and Young's inequalities from this estimate.
\end{proof}
\subsection{Interior estimates} \label{subsec:int}
\begin{theorem}\label{theo:int-comp}
Let $u \in W^{1,p}(\Omega)$ be a solution to~\eqref{eq:diveq}. Assume that $x_0 \in \Omega$ and $R>0$ such that $B_{4R}(x_0) \subset \Omega$. Assume moreover that $A$ satisfies a $(\delta,R_0)$-BMO condition for some constants $\delta\in(0,1)$. Then there exists $v \in L^{\infty}(B_{R}(x_0))\cap W^{1,p}(B_{2R}(x_0))$ and $m>0$ such that two following inequalities
\begin{align}\label{eq:theoIc}
\|\nabla v\|^p_{L^{\infty}(B_{R}(x_0))} & \le C\fint_{B_{4R}(x_0)} |\nabla u|^p  dx + C  \fint_{B_{4R}(x_0)}  |f|^{\frac{p}{p-1}} + |\nabla g|^p dx,
\end{align}
and
\begin{align}\label{eq:theoId}
\fint_{B_{2R}(x_0)} |\nabla u - \nabla v|^pdx \le C(\delta^m+\gamma) \fint_{B_{4R}(x_0)} |\nabla u|^p dx  + C_{\gamma}   \fint_{B_{4R}(x_0)} |f|^{\frac{p}{p-1}} + |\nabla g|^p dx,
\end{align} 
hold for every $\gamma \in (0,1)$.
\end{theorem}
\begin{proof}
Let us fix a point $x_0 \in \Omega$ and $R>0$ such that $B_{4R}(x_0) \subset \Omega$. For simplicity of notation let us set $B_{4R}:= B_{4R}(x_0)$, $B_{2R}:= B_{2R}(x_0)$ and $B_{R}:= B_{R}(x_0)$. The proof will be achieved in two  steps that corresponding to two level comparisons with homogeneous problems.\\ 

\textit{Step 1:} We consider the unique solution $w$ to the homogeneous problem
\begin{equation}\label{eq:I1}
\begin{cases} \mathrm{div} A(x,\nabla w) & = \ 0, \qquad \, \mbox{ in } B_{4R},\\ 
\hspace{1.2cm} w & = \ u - g. \ \, \mbox{ on } \partial B_{4R}.\end{cases}
\end{equation}
Let us first prove that the following comparison estimate 
\begin{align}\label{eq:lem1b}
\fint_{B_{4R}} |\nabla u - \nabla w|^p dx  & \le   \gamma \fint_{B_{4R}} |\nabla u|^{p} dx  + C(\gamma) \fint_{B_{4R}}  |f|^{\frac{p}{p-1}} + |\nabla g|^p dx
\end{align}
holds for every $\gamma \in (0,1)$. By choosing $u - w - \overline{g}$ as a test function of equations~\eqref{eq:diveq} and~\eqref{eq:I1}, where $\overline{g} = g$ in $\overline{B_{4R}}$, one obtains that
\begin{align}\label{eq:L1I1}
& \int_{B_{4R}} \langle \left(A(x,\nabla u) - A(x, \nabla w)\right), \nabla (u - w) \rangle dx  = \int_{B_{4R}} \langle \left(A(x,\nabla u) - A(x, \nabla w)\right), \nabla g \rangle dx \notag \\ & \hspace{5cm} + \int_{B_{4R}} \langle f, \nabla (u - w) \rangle dx  - \int_{B_{4R}} \langle f, \nabla g \rangle dx.
\end{align}
Moreover, we notice that since two conditions~\eqref{cond:A1} and~\eqref{cond:A2} of $A$, it can be deduced from~\eqref{eq:L1I1} that there exists a positive constant $C$ depending on $\Lambda_1, \Lambda_2$ such that
\begin{align}\label{eq:L1I2-a}
& \int_{B_{4R}} \left(|\nabla u|^2 + |\nabla w|^2 \right)^{\frac{p-2}{2}}|\nabla u - \nabla w|^2  dx \le C \left(\int_{B_{4R}} |\nabla u|^{p-1} |\nabla g| dx \right. \notag \\ & \hspace{2cm}  \left. + \int_{B_{4R}} |\nabla w |^{p-1} |\nabla g| dx   + \int_{B_{4R}} |f| |\nabla u - \nabla w| dx + \int_{B_{4R}} |f|  |\nabla g| dx\right).
\end{align}
The fundamental inequality
\begin{align*}
|\nabla w|^{p-1} & \le \left(|\nabla u| + |\nabla u - \nabla w|\right)^{p-1} \le 2^p \left(|\nabla u|^{p-1} + |\nabla u - \nabla w|^{p-1}\right),
\end{align*}
yields from~\eqref{eq:L1I2-a} that 
\begin{align}\label{eq:L10}
& \int_{B_{4R}} \left(|\nabla u|^2 + |\nabla w|^2 \right)^{\frac{p-2}{2}}|\nabla u - \nabla w|^2  dx \le  C  \left(\int_{B_{4R}} |\nabla u|^{p-1} |\nabla g| dx  \right. \notag \\  
 & \hspace{2cm} \left. + \int_{B_{4R}} |\nabla u - \nabla w |^{p-1} |\nabla g| dx + \int_{B_{4R}} |f| |\nabla u - \nabla w| dx + \int_{B_{4R}} |f|  |\nabla g| dx\right).
\end{align}
In order to estimate the right-hand side of~\eqref{eq:L1I2-a}, we are allowed to apply the  inequality~\eqref{eq:estHY} which is a consequence of Young's inequality. More precisely, it is known that for any $\epsilon \in (0,1)$, there exists $m(\epsilon) = m(p,\epsilon)>0$ such that
\begin{align}\label{eq:estHY}
|a|^{p-1}|b| \le \epsilon |a|^p + \epsilon^{1-p} |b|^p \ \mbox{ or } \ |ab| \le \epsilon |a|^{p} + \epsilon^{\frac{1}{1-p}} |b|^{\frac{p}{p-1}}.
\end{align}
It is straightforward to obtain these following inequalities:
\begin{align}\label{eq:estI0}
\int_{B_{4R}} |\nabla u|^{p-1} |\nabla g| dx  \le \epsilon \int_{B_{4R}} |\nabla u|^{p} dx + \epsilon^{1-p} \int_{B_{4R}}  |\nabla g|^p dx,
\end{align}
\begin{align}\label{eq:estI1}
\int_{B_{4R}} |\nabla u - \nabla w |^{p-1} |\nabla g| dx   \le \epsilon  \int_{B_{4R}} |\nabla u- \nabla w |^{p} dx + \epsilon^{1-p} \int_{B_{4R}} |\nabla g|^p dx,
\end{align}
\begin{align}\label{eq:estI2}
\int_{B_{4R}} |f| |\nabla u - \nabla w| dx   \le  \epsilon \int_{B_{4R}} |\nabla u - \nabla w|^p dx + \epsilon^{\frac{1}{1-p}} \int_{B_{4R}} |f|^{\frac{p}{p-1}} dx,
\end{align}
and
\begin{align}\label{eq:estI3}
\int_{B_{4R}} |f|  |\nabla g| dx \le \int_{B_{4R}} |\nabla g|^p dx +  \int_{B_{4R}} |f|^{\frac{p}{p-1}} dx.
\end{align}
Substituting all estimates in~\eqref{eq:estI0}-\eqref{eq:estI3} into~\eqref{eq:L10}, one finds
\begin{align}\label{eq:L500}
& \int_{B_{4R}} \left(|\nabla u|^2 + |\nabla w|^2 \right)^{\frac{p-2}{2}}|\nabla u - \nabla w|^2  dx \le C\left( \epsilon \int_{B_{4R}} |\nabla u|^{p} dx + 2\epsilon \int_{B_{4R}} |\nabla u - \nabla w|^p dx \right.  \notag \\  
 & \hspace{6cm} \left. + 3\epsilon^{-m} \int_{B_{4R}} |f|^{\frac{p}{p-1}} + |\nabla g|^p dx\right).
\end{align}
where $m = \max\left\{\frac{1}{p-1};p-1\right\}$. The first case, when $p \ge 2$, let us apply the fundamental inequality
\begin{align*}
|\nabla u - \nabla w|^p \le C\left(|\nabla u|^2 + |\nabla w|^2 \right)^{\frac{p-2}{2}}|\nabla u - \nabla w|^2, 
\end{align*}
and replace $4C\epsilon$ in~\eqref{eq:L500} by $\gamma \in (0,1)$ to deduce~\eqref{eq:lem1b}. In the other case when $1<p<2$, one may use the following decomposition
\begin{align}\label{fund-ineq-2}
|\nabla u - \nabla w|^p & = \left(|\nabla u|^2 + |\nabla w|^2 \right)^{\frac{p(2-p)}{4}}\left[\left(|\nabla u|^2 + |\nabla w|^2 \right)^{\frac{p-2}{2}}|\nabla u - \nabla w|^2\right]^{\frac{p}{2}} \notag \\
& = C\left(|\nabla u|^p + |\nabla u-\nabla w|^p \right)^{1-\frac{p}{2}}\left[\left(|\nabla u|^2 + |\nabla w|^2 \right)^{\frac{p-2}{2}}|\nabla u - \nabla w|^2\right]^{\frac{p}{2}} \notag \\
& \le \varepsilon \left(|\nabla u|^p + |\nabla u-\nabla w|^p\right) + C_{\varepsilon}\left(|\nabla u|^2 + |\nabla w|^2 \right)^{\frac{p-2}{2}}|\nabla u - \nabla w|^2, 
\end{align}
for every $\varepsilon >0$. In order to conclude~\eqref{eq:lem1b}, it allows us to choose suitable values of $\varepsilon$ and $\epsilon$ depending on $\gamma \in (0,1)$.\\

\textit{Step 2:} Let $v$ be the unique solution to the following problem 
\begin{equation}\label{eq:I2}
\begin{cases} \mathrm{div} \overline{A}_{B_{2R}}(\nabla v) & = \ 0, \ \ \mbox{ in } B_{2R},\\ 
\hspace{1.2cm} v & = \ w, \ \mbox{ on } \partial B_{2R}.\end{cases}
\end{equation}
The basic regularity for solution to~\eqref{eq:I2} gives us
\begin{align} \label{eq:theoId_2}
\|\nabla v\|^p_{L^{\infty}(B_{R})} &\le C \fint_{B_{2R}} |\nabla v|^p dx   \le C \fint_{B_{2R}} |\nabla w|^p dx. 
\end{align}
Thanks to~\eqref{eq:lem1b}, for every $\gamma\in (0,1)$ one has the following estimate
\begin{align} \label{eq:theoId_1}
\fint_{B_{2R}} |\nabla w|^p dx & \le  C\fint_{B_{4R}} |\nabla u|^p  dx +  C\fint_{B_{4R}} |\nabla u - \nabla w|^p dx \notag \\ 
& \le C(1+\gamma) \fint_{B_{4R}} |\nabla u|^p  dx +  C \left( \fint_{B_{4R}}  |f|^{\frac{p}{p-1}} + |\nabla g|^p dx  \right),
\end{align}
which implies to~\eqref{eq:theoIc} from~\eqref{eq:theoId_2}. On the other hand, by taking $v - w \in W_0^{1,p}(B_{2R})$ as the test function to~\eqref{eq:I2} and~\eqref{eq:I1}, there holds
\begin{align*}
& \fint_{B_{2R}} \langle (\overline{A}_{B_{2R}}(\nabla w) - \overline{A}_{B_{2R}}(\nabla v)), (\nabla v - \nabla w) \rangle dx \\
& \hspace{4cm} = \fint_{B_{2R}} \langle (\overline{A}_{B_{2R}}(\nabla w) - A(x,\nabla w)), (\nabla v - \nabla w) \rangle dx.
\end{align*}
Thanks to both~\eqref{cond:A2} and Young inequality, for every $\varepsilon \in (0,1)$ one gets that
\begin{align}\label{eq:B100}
& \fint_{B_{2R}} \left(|\nabla v|^2 + |\nabla w|^2 \right)^{\frac{p-2}{2}}|\nabla v - \nabla w|^2  dx \le C \fint_{B_{2R}} \theta(x) |\nabla w|^{p-1} |\nabla v - \nabla w| dx \notag \\
& \hspace{3cm} \le \varepsilon \fint_{B_{2R}} |\nabla v - \nabla w|^p dx + C_{\varepsilon}  \fint_{B_{2R}} [\theta(x)]^{\frac{p}{p-1}} |\nabla w|^{p} dx, 
\end{align}
where the function $\theta$ is defined by
\begin{align}\notag
\theta(x) = \sup_{\xi \in \mathbb{R}^n \setminus \{0\}} \frac{\left|A(x,\xi)-\overline{A}_{B_{2R}}(\xi)\right|}{|\xi|^{p-1}}.
\end{align}
It is well-known that the higher integrability of~\eqref{eq:I1}, there exists $P_0>p$ such that
\begin{align}\notag
\left(\fint_{B_{2R}} |\nabla w|^{P_0} dx\right)^{\frac{1}{P_0}} \le C \left(\fint_{B_{4R}} |\nabla w|^{p} dx\right)^{\frac{1}{p}}. 
\end{align}
Using this inequality and the fact that $\theta$ is bounded by $2\Lambda_1$ under assumption~\eqref{cond:A1}, it follows that
\begin{align}\label{eq:B200}
\fint_{B_{2R}} [\theta(x)]^{\frac{p}{p-1}} |\nabla w|^{p} dx & \le \left(\fint_{B_{2R}} |[\theta(x)]^{\frac{p}{p-1}}|^{\frac{P_0}{P_0-p}} dx\right)^{\frac{P_0-p}{P_0}} \left(\fint_{B_{2R}} |\nabla w|^{P_0} dx\right)^{\frac{p}{P_0}} \notag \\
& \le \left(\fint_{B_{2R}} \theta(x) dx\right)^{\frac{P_0-p}{P_0}} \left(\fint_{B_{2R}} |\nabla w|^{P_0} dx\right)^{\frac{p}{P_0}} \notag \\
& \le C \delta^m \fint_{B_{4R}} |\nabla w|^{p} dx, 
\end{align}
where $m = \frac{P_0-p}{P_0}$. Here we use the small-BMO condition $[A]_{R_0} \le \delta$ and apply to the last estimate in~\eqref{eq:B200}. Similar to the previous proof, from~\eqref{eq:B100} and~\eqref{eq:B200} one can find a positive constant $\tilde{C}=\tilde{C}(n,p,\Lambda_1,\Lambda_2)$ such that 
\begin{equation}\label{eq:lem2a}
\fint_{B_{2R}} |\nabla w - \nabla v|^p dx \le \tilde{C} \delta^m \fint_{B_{4R}} |\nabla w|^p dx.
\end{equation}
We refer the reader to~\cite{MP11} for the proof of~\eqref{eq:lem2a}. This estimate gives us
\begin{align*}
\fint_{B_{2R}} |\nabla u - \nabla v|^p dx &\le C\fint_{B_{2R}} |\nabla u - \nabla w|^p dx + C \fint_{B_{2R}} |\nabla w - \nabla v|^p dx  \\
& \le C\fint_{B_{4R}} |\nabla u - \nabla w|^p dx + \tilde{C}\delta^m \fint_{B_{4R}} |\nabla w|^p dx.
\end{align*}
This establishes our conclusion~\eqref{eq:theoId} by using the results from both~\eqref{eq:lem1b} and~\eqref{eq:theoId_1}.
\end{proof}

\subsection{Boundary estimates}\label{subsec:boundary}

We here treat the boundary case, where the boundary $\partial\Omega$ satisfies the local flatness in sense of Reifenberg, the proof can be done in the same argument as that of interior case. The precise statement can be found in Theorem \ref{theo:boundary-comp} as follows.
\begin{theorem}\label{theo:boundary-comp}
Let $u \in W^{1,p}(\Omega)$ be a solution to~\eqref{eq:diveq} and $\Omega$ be a $(\delta,R_0)$-Reifenberg flat domain with $\delta \in (0,1/2]$. Assume that $x_0 \in \partial\Omega$, $0<R<R_0/4$ and $A$ satisfies a small-BMO condition $[A]_{R_0} \le \delta$. Then there exists $\tilde{v} \in L^{\infty}(\Omega_{R/9})\cap W^{1,p}(\Omega_{R/9})$ and $m>0$ such that two following inequalities
\begin{align}\label{eq-theo:boundary-1}
\|\nabla \tilde{v}\|^p_{L^{\infty}(\Omega_{R/9})} & \le C\fint_{\Omega_{4R}} |\nabla u|^p  dx + C  \fint_{\Omega_{4R}}  |f|^{\frac{p}{p-1}} + |\nabla g|^p dx,
\end{align}
and
\begin{align}\label{eq-theo:boundary-2}
\fint_{\Omega_{R/9}} |\nabla u - \nabla \tilde{v}|^pdx \le C(\delta^m+\gamma) \fint_{\Omega_{4R}} |\nabla u|^p dx  + C_{\gamma} \fint_{\Omega_{4R}} |f|^{\frac{p}{p-1}} + |\nabla g|^p dx,
\end{align} 
hold for every $\gamma \in (0,1)$. Here we denote $\Omega_{\varrho} = B_{\varrho}(x_0)\cap \Omega$ for every $\varrho>0$.
\end{theorem}
\begin{proof}
Let $w$ be the unique solution to the following problem
 \begin{equation}\label{eq:B1}
\begin{cases} \mathrm{div} A(x,\nabla w) & = \ 0, \hspace{.85cm} \mbox{ in } \Omega_{4R},\\ 
\hspace{1.2cm} w & = \ u - g, \ \mbox{ on } \partial \Omega_{4R}.\end{cases}
\end{equation}
By the same technique as in the proof of~\eqref{eq:lem1b} in Theorem~\ref{theo:int-comp}, for every $\gamma \in (0,1)$ one obtains that
\begin{align}\label{est_lem:u-w500}
\fint_{\Omega_{4R}} |\nabla u - \nabla w|^p dx  & \le \gamma \fint_{\Omega_{4R}} |\nabla u|^{p} dx + C \fint_{\Omega_{4R}} |f|^{\frac{p}{p-1}} + |\nabla g|^p dx,
\end{align} 
which deduces to
\begin{align} \label{est_lem:w}
\fint_{\Omega_{4R}} |\nabla w|^p dx & \le C \fint_{\Omega_{4R}} |\nabla u|^p dx + C \fint_{\Omega_{4R}} |\nabla u - \nabla w|^p dx\notag \\ 
& \le  C \fint_{\Omega_{4R}} |\nabla u|^p dx + C \left(\fint_{\Omega_{4R}}  |f|^{\frac{p}{p-1}} + |\nabla g|^p dx  \right).
\end{align}
On the other hand, since $\Omega$ is a $(\delta,R_0)$-Reifenberg flat domain with $\delta \in (0,1/2]$, there exists a coordinate system $\{y_1, y_2, ..., y_n\}$ with the origin $0 \in \Omega$ such that in this coordinate system $x_0 = - r\delta/(1-\delta)y_n$ and
\begin{align}\label{cond:B}
B_{r}^+(0) \subset B_r(0) \cap \Omega \subset B_r(0) \cap \{y_n> -2 r\delta/(1-\delta)\} \subset  B_r(0) \cap \{y_n > -4r\delta\},
\end{align}
where  $r = R(1-\delta)$ and $B_{r}^+(0) = B_r(0) \cap \{y_n > 0\}$. Therefore, one can find $\delta_0>0$ small enough such that for all $\delta \in (0,\delta_0)$ then condition~\eqref{cond:B} and the following estimate are valid
\begin{align}\label{cond:B2}
B_{R/9}(x_0) \subset B_{r/8}(0) \subset B_{r/4}(0) \subset B_r(0) \subset B_{4r}(x_0) \subset B_{4R}(x_0).
\end{align}
Let $v$ be the unique solution to the equation
\begin{equation}\label{eq:B3}
\begin{cases} \mathrm{div} \overline{A}_{B_r(0)}(\nabla v) & = \ 0, \ \, \mbox{ in } B_r(0) \cap \Omega,\\ 
\hspace{1.2cm} v & = \ w, \ \mbox{ on } \partial (B_r(0) \cap \Omega).\end{cases}
\end{equation} 
Similar to~\eqref{eq:lem2a}, one may obtain from~\eqref{cond:B2} the following estimate
\begin{align}\label{est_lem:v-w500}
\fint_{B_{r/8}(0)\cap \Omega} |\nabla v - \nabla w|^pdx \le C \fint_{B_{r}(0)\cap \Omega} |\nabla v - \nabla w|^pdx \le C \delta^m \fint_{\Omega_{4R}} |\nabla w|^pdx.
\end{align}
Because of the fact that $L^{\infty}$-norm of $\nabla v$ up to the boundary may not exist if $\partial \Omega$ is not regular enough, we consider $\tilde{v}$ as the weak solution to an another problem
\begin{equation}\label{eq:B4}
\begin{cases} \mathrm{div} \overline{A}_{B_r(0)}(\nabla \tilde{v}) & = \ 0, \ \mbox{ in } B_r^{+}(0),\\ 
\hspace{1.2cm} \tilde{v} & = \ 0, \ \mbox{ on }  B_r(0)\cap \{y_n=0\}.\end{cases}
\end{equation}
By the same technique as the proof of Theorem~\ref{theo:int-comp}, one can find $m>0$ such that for every $\delta \in (0,\delta_0)$, the weak solution $\tilde{v}$ of~\eqref{eq:B4} satisfies the following estimates
\begin{align}\label{est_lem:B5a}
\|\nabla \tilde{v}\|^p_{L^{\infty}(B_{r/4}(0))} &\le C \fint_{B_r(0)}|\nabla v|^p dx, 
\end{align}
and
\begin{align} \label{est_lem:B5b}
\fint_{B_{r/8}(0)} |\nabla v - \nabla U|^p dx & \le C\delta^m \fint_{B_r(0)} |\nabla v|^pdx.
\end{align}
Combining~\eqref{est_lem:B5a} and the regularity of solution $v$ to~\eqref{eq:B3} with notice~\eqref{cond:B2}, one gets that
\begin{align*}
\|\nabla \tilde{v}\|^p_{L^{\infty}(\Omega_{R/9})}  \le C \fint_{B_{r}(0)} |\nabla v|^p dx \le C \fint_{B_{r}(0)} |\nabla w|^p dx  \le C \fint_{\Omega_{4R}} |\nabla w|^p dx,
\end{align*}
which follows to~\eqref{eq-theo:boundary-1} from~\eqref{est_lem:w}. Moreover, ones also obtains from~\eqref{cond:B2} that
\begin{align}\label{est_lem:final}
\fint_{\Omega_{R/9}} |\nabla u - \nabla \tilde{v}|^pdx & \le C \fint_{B_{r/8}(0)} |\nabla u - \nabla \tilde{v}|^pdx \notag \\ 
& \le C \fint_{B_{r/8}(0)} |\nabla u - \nabla w|^pdx + C \fint_{B_{r/8}(0)} |\nabla w - \nabla v|^pdx \notag \\  
& \qquad + C \fint_{B_{r/8}(0)} |\nabla v - \nabla \tilde{v}|^pdx.
\end{align}
In order to obtain~\eqref{eq-theo:boundary-2}, one only takes into account all estimates in~\eqref{est_lem:u-w500}, \eqref{est_lem:w}, \eqref{est_lem:v-w500} and ~\eqref{est_lem:B5b} into the right-hand side of~\eqref{est_lem:final}. 
\end{proof}

\section{Gradient estimates}\label{sec:main}

We devote this section to proving our main results in theorems \ref{theo:regularity} and \ref{theo:I-alpha}. Before providing the proofs of them, let us state and prove the theorem of `good-$\lambda$' inequality (Theorem \ref{theo:goodlambda}), which plays a significant role in our main proofs. 

\subsection{Good-$\lambda$ inequality}
\begin{theorem}\label{theo:goodlambda}
Let $p>1$, $\alpha \in [0,n)$, $\omega \in \mathcal{A}_\infty$, $f \in L^{\frac{p}{p-1}}(\Omega; \mathbb{R}^n)$, $g \in W^{1,p}(\Omega;\mathbb{R})$ and $u$ be a weak solution to \eqref{eq:diveq}. For any $\varepsilon>0$, $\lambda>0$ and $R_0>0$, there exist some constants $\delta = \delta(n,\varepsilon,[\omega]_{\mathcal{A}_\infty})$, ${\sigma} = \sigma(n,p,\alpha)$ and $\kappa = \kappa(n,p,\alpha,\varepsilon, [\omega]_{\mathcal{A}_\infty}, \mathrm{diam}(\Omega)/R_0)$ such that if $\Omega$ is a $(\delta,R_0)$-Reifenberg flat domain satisfying $[A]_{R_0} \le \delta$, then
\begin{align}\label{eq:lambdabound}
&\omega\left(\{\mathbf{M}\mathbf{M}_{\alpha}(|\nabla u|^{p})>{\sigma}\lambda, \mathbf{M}_{\alpha}(|f|^{\frac{p}{p-1}}+|\nabla g|^p) \le \kappa\lambda \}\cap \Omega \right) \notag \\ & \hspace{6cm}\leq C \varepsilon \omega\left(\{\mathbf{M} \mathbf{M}_{\alpha}(|\nabla u|^{p})> \lambda\}\cap \Omega \right).
\end{align}
Here, we note that the constant $C$ depends only on $n,p,\alpha, \varepsilon,\mathrm{diam}(\Omega)/R_0,[\omega]_{\mathcal{A}_\infty}$.
\end{theorem}

To obtain the proof of Theorem \ref{theo:goodlambda}, we require the following Lemma \ref{lem:wReig}, the very important ingredient. The utility of this lemma normally relies on the Vitali type covering lemma (to cover a set $G$ by a countable family of pairwise disjoint closed balls) and the Lebesgue differentiation theorem (to control the size of the set on which the integral average can be large in terms of the $L^1$-norm), that are widespread in harmonic analysis. It refers to \cite{Vitali08}, the famous result in measure theory of Euclidean spaces, noticed by Vitali and later various literature concerning its modifications and applications \cite{CP1998,Wang2}.  The use of Vitali's covering lemma combining with maximal function techniques was first introduced by Duzaar and  Mingione in \cite{Duzamin2, 55DuzaMing}. Further, several references \cite{SSB4,BW1_1,MP11} are also worth to read in solution estimates for elliptic and parabolic equations/systems.


\begin{lemma}[Covering lemma]\label{lem:wReig}
Let $\omega \in \mathcal{A}_\infty$ and $\Omega$ be a $(\delta,R_0)$-Reifenberg flat domain for some $\delta \in (0,1)$. Suppose that the sequence of balls $\{B_r(z_i)\}_{i=1}^N$ with center $z_i \in \bar{\Omega}$ and radius $r \le R_0/10$ covers $\Omega$. Let $V \subset W \subset \Omega$ be measurable sets for which there exists $0<\varepsilon<1$ such that:
\begin{itemize}
\item[(i)] $\omega(V) \le \varepsilon \omega(B_r(z_i))$ for all $i=1,2,...,N$;
\item[(ii)] $\omega(V \cap B_\rho(x)) \ge \varepsilon\omega(B_\rho(x)) \Rightarrow B_\rho(x) \cap \Omega \subset W$, for all $x \in \Omega, \rho \in (0,2r]$.
\end{itemize}
Then, there exists a constant $C = C(n,[\omega]_{\mathcal{A}_\infty})$ such that $\omega(V) \le C\varepsilon \omega(W)$.
\end{lemma}

\begin{proof}[Proof of Theorem \ref{theo:goodlambda}]
To do this we use a technique similar to the one in \cite{55QH4} for the problem with measure data, but here we confine ourselves to improve and modify to the proof of a version involving $\mathbf{M}_\alpha$ in the context of problem \eqref{eq:diveq}. Let us now consider two sets 
\begin{equation*}
V = \left\{\mathbf{M}\mathbf{M}_{\alpha}(|\nabla u|^{p})>{\sigma}\lambda, \mathbf{M}_{\alpha}(|f|^{\frac{p}{p-1}}+|\nabla g|^p) \le \kappa\lambda \right\}\cap \Omega, 
\end{equation*}
and
\begin{equation*}
W = \{\mathbf{M}\mathbf{M}_{\alpha}(|\nabla u|^{p})> \lambda\}\cap \Omega,
\end{equation*}
for any $\lambda>0$, where the constants ${\sigma}$, $\kappa$ in these sets will be specified later. Once having the Lemma \ref{lem:wReig} at hand, we outline the main steps that need to prove the sets $V$, $W$ satisfying all the assumptions, i.e., for any $\varepsilon>0$ there holds
\begin{enumerate}
\item[(i)] $\omega(V) \le \varepsilon \omega(B_{R_0}(0))$,
\item[(ii)] for all $x \in Q = B_{2D_0}(x_0)$, $r \in (0,2R_0]$, if $\omega(V\cap B_r(x)) \ge
C \varepsilon \omega(B_r(x))$ then $B_r(x) \cap Q \subset W$, where $D_0 = \mathrm{diam}(\Omega)$.
\end{enumerate}
We outline the main steps in the proof following above conditional lemma. More precisely, we first show that $(i)$ holds. Without loss of generality, we may assume that $V \neq \emptyset$, then there exists $x_1 \in \Omega$ such that
\begin{equation}\label{eq:E1}
\mathbf{M}_{\alpha}(|f|^{\frac{p}{p-1}} + |\nabla g|^p)(x_1) \le \kappa \lambda.
\end{equation}
By Lemma~\ref{lem:boundM}, the boundedness property of maximal function $\mathbf{M}$ from $L^1(\Omega)$ to $L^{1,\infty}(\Omega)$ gives
\begin{equation}\label{eq:E4}
|V| \le \left|\left\{\mathbf{M}\mathbf{M}_{\alpha}(|\nabla u|^{p})>{\sigma}\lambda \right\} \cap \Omega\right|  \le \frac{1}{{\sigma}\lambda} \int_{\Omega}\mathbf{M}_{\alpha}\left(|\nabla u|^p\right) dx.
\end{equation}
Applying the boundedness property of fractional maximal function $\mathbf{M}_{\alpha}$ in Lemma~\ref{lem:Malpha}, there holds
\begin{align*}
 \int_{\Omega} \mathbf{M}_\alpha(|\nabla u|^p) dx&= \int_{0}^{\infty} \left|\left\{x\in \Omega: \ \mathbf{M}_\alpha(|\nabla u|^p)(x)>\lambda\right\}\right| d\lambda\\
 & \leq  CD_0^n\lambda_0+\int_{\lambda_0}^{\infty} \left|\left\{x \in \Omega: \ \mathbf{M}_\alpha(|\nabla u|^p)(x)>\lambda\right\}\right| d\lambda\\
 &\leq C D_0^n\lambda_0+ C \left(\int_{\Omega}|\nabla u|^pdx\right)^{\frac{n}{n-\alpha}}\int_{\lambda_0}^{\infty}\lambda^{-\frac{n}{n-\alpha}}d\lambda\\
 &= CD_0^n\lambda_0+ C \left(\int_{\Omega}|\nabla u|^pdx\right)^{\frac{n}{n-\alpha}}\lambda_0^{-\frac{\alpha}{n-\alpha}},
\end{align*}
for any $\lambda_0>0$. In this formula, let us choose $\lambda_0=D_0^{-n+\alpha}\int_{\Omega}|\nabla u|^pdx$,
to follow that
\begin{align}\label{eq:E3}
\int_{\Omega} \mathbf{M}_{\alpha}\left(|\nabla u|^p\right) dx \le C_1 D_0^{\alpha} \int_{\Omega} |\nabla u|^p dx.
\end{align}
For the sake of readability, in some inequalities follow, the constants $C_i$ appearing might vary and must be indicated precisely. As such, this makes sense when we choose the value $\varepsilon>0$ in the statement of theorem at the end of proof depends only on a specific final constant. Plugging the validity of \eqref{eq:E3} to~\eqref{eq:E4} and~\eqref{eq:prop1} from Proposition~\ref{prop1} to infer that
\begin{align}\label{eq:E6}
|V|  &\le \frac{C_2D_0^{\alpha}}{{\sigma}\lambda} \int_{\Omega} |\nabla g|^p + |f|^{\frac{p}{p-1}} dx   \le  \frac{C_2D_0^{\alpha}}{{\sigma}\lambda} \int_{B_{D_0}(x_1)} |\nabla g|^p + |f|^{\frac{p}{p-1}} dx. 
\end{align}
Thanks to~\eqref{eq:E1}, it deduces from~\eqref{eq:E6} that
\begin{align}\label{eq:E7}
|V| \le \frac{C_2D_0^{n}}{{\sigma}\lambda} \mathbf{M}_{\alpha}(|\nabla g|^p + |f|^{\frac{p}{p-1}})(x_1)  \le \frac{C_2D_0^{n}}{{\sigma}\lambda} \kappa \lambda \le \frac{C_3 \kappa}{{\sigma}} |B_{R_0}(0)|.
\end{align}
In view of the definition of Muckenhoupt weight $\mathcal{A}_{\infty}$, we get by~\eqref{eq:E7} that
\begin{align*}
\omega(V) \le C_4 \left(\frac{|V|}{|B_{R_0}(0)|}\right)^{\nu} \omega\left(B_{R_0}(0)\right)  \le {C_5}{({\sigma})}^{-\nu} \kappa^\nu  \omega\left(B_{R_0}(0)\right) \le \varepsilon \omega\left(B_{R_0}(0)\right),
\end{align*}
where $\kappa$ is small enough satisfying ${C_5}{({\sigma})}^{-\nu} \kappa^\nu < \varepsilon$, we then immediately obtain $(i)$.\\

Let $x \in \Omega$, $r \in (0,2R_0]$ and $\lambda>0$, the remainder will be dedicated to the proof of $(ii)$, and the proof performed via a contradiction. Let us assume that $V \cap B_r(x) \neq \emptyset$ and $B_r(x) \cap \Omega \cap W^c \neq \emptyset$, i.e., there exist $x_2, x_3 \in B_r(x) \cap \Omega$ such that 
\begin{equation}\label{eq:M}
\mathbf{M}\mathbf{M}_{\alpha}(|\nabla u|^p)(x_2) \le \lambda \quad \mbox{ and } \quad \mathbf{M}_{\alpha}(|f|^{\frac{p}{p-1}}+|\nabla g|^p)(x_3) \le \kappa \lambda.
\end{equation}
We will show that 
\begin{equation}\label{eq:F1}
\omega(V \cap B_r(x)) < \varepsilon \omega(B_r(x)),
\end{equation}
which is a contradiction by Lemma~\ref{lem:wReig}. Indeed, for any $y \in B_r(x)$, it is easy to see that 
$$ B_{\rho}(y) \subset  B_{\rho+r}(x) \subset B_{\rho+2r}(x_2) \subset B_{3\rho}(x_2), \ \mbox{ for all } \ \rho \ge r,$$ 
which follows from \eqref{eq:M} that
\begin{align}\label{eq:T1}
\sup_{\rho\ge r}\fint_{B_{\rho}(y)}\mathbf{M}_{\alpha}(|\nabla u|^p) dx & \le 3^n  \sup_{\rho\ge r}\fint_{B_{3\rho}(x_2)}\mathbf{M}_{\alpha}(|\nabla u|^p)dx \le 3^n \mathbf{M}\mathbf{M}_{\alpha}(|\nabla u|^p)(x_2) \le 3^n \lambda.
\end{align}
Similarly, for any $y \in B_r(x)$, for all $0<\rho<r$ and  $z \in B_\rho(y)$, since $B_{\varrho}(z) \subset B_{4r}(x_2)$ for any $\varrho \ge r$, we also obtain that
\begin{align}\label{eq:T2}
\sup_{0<\rho<r} \fint_{B_\rho(y)} \left( \sup_{\varrho\ge r} \varrho^{\alpha} \fint_{B_{\varrho}(z)} |\nabla u|^p \right) dz \le 4^{n-\alpha} \mathbf{M}\mathbf{M}_{\alpha} (|\nabla u|^p)(x_2)  \le 4^{n-\alpha}\lambda.
\end{align}
Moreover, by the definitions of $\mathbf{M}$ and $\mathbf{M}_\alpha$, we can conclude from~\eqref{eq:T1} and~\eqref{eq:T2} that
\begin{align}\nonumber
\mathbf{M}\mathbf{M}_{\alpha}(|\nabla u|^p)(y) & \le \max\left\{\sup_{0<\rho<r} \fint_{B_\rho(y)} \left( \sup_{0<\varrho< r} \varrho^{\alpha} \fint_{B_{\varrho}(z)} |\nabla u|^p \right) dz; \right. \\ \nonumber
& \left. \sup_{0<\rho<r} \fint_{B_\rho(y)} \left( \sup_{\varrho\ge r} \varrho^{\alpha} \fint_{B_{\varrho}(z)} |\nabla u|^p \right) dz;  \quad \sup_{\rho\ge r}\fint_{B_{\rho}(y)}\mathbf{M}_{\alpha}(|\nabla u|^p) dx \right\} \\ \label{eq:T3}
& \le \max\left\{\sup_{0<\rho<r} \fint_{B_\rho(y)} \left( \sup_{0<\varrho< r} \varrho^{\alpha} \fint_{B_{\varrho}(z)} |\nabla u|^p \right) dz; \ 4^n \lambda \right\},
\end{align}
for all $y \in B_r(x)$. Thanks to~\eqref{eq:MrMr} in Lemma~\ref{lem:MrMr}, we obtain that from~\eqref{eq:T3}, it provides
\begin{align}\label{eq:T4}
\mathbf{M}\mathbf{M}_{\alpha}(|\nabla u|^p)(y)  \le \max\left\{\mathbf{M}^{2r}_{\alpha} (\chi_{B_{2r}(x)}|\nabla u|^p) (y); \ 4^n \lambda \right\}.
\end{align}
Hence if $\sigma$ is chosen satisfying ${\sigma} > 4^{n}$, then for any $\lambda>0$, by~\eqref{eq:T4} there holds
\begin{align}\label{eq:F2}
|V \cap B_r(x)| \le \left|\left\{\mathbf{M}^{2r}_{\alpha} (\chi_{B_{2r}(x)}|\nabla u|^p)  > {\sigma} \lambda\right\} \cap B_r(x) \cap \Omega\right|.
\end{align}
We remark that if $\overline{B}_{8r}(x) \subset \mathbb{R}^n\setminus \Omega$ then $V \cap B_r(x) = \emptyset$. So we need to consider two cases: $x$ is in the interior domain $B_{8r}(x) \subset \Omega$ and $x$ is near the boundary $B_{8r}(x) \cap \partial \Omega \neq \emptyset$. And the proof in each case consists in matching the comparison estimates of Lemmas and Theorems in the interior domain and on the boundary.\\

Let us now consider the first case $B_{8r}(x) \subset \Omega$. Thanks to Theorem~\ref{theo:int-comp}, under small-BMO condition of $A$ one can find $v \in L^{\infty}(B_{2r}(x))\cap W^{1,p}(B_{4r}(x))$ and $m>0$ such that
\begin{align}\label{est:Ia}
\|\nabla v\|^p_{L^{\infty}(B_{2r}(x))} \le C_6 \fint_{B_{8r}(x)} |\nabla u|^p dx + C_6 \fint_{B_{8r}(x)} |f|^{\frac{p}{p-1}} + |\nabla g|^p dx, 
\end{align}
and for every $\gamma \in (0,1)$ there holds
\begin{align}\label{est:Ib}
\fint_{B_{4r}(x)} |\nabla u - \nabla v|^pdx \le C_6 (\delta^m+\gamma) \fint_{B_{8r}(x)} |\nabla u|^p dx + C_6 \fint_{B_{8r}(x)} |f|^{\frac{p}{p-1}} + |\nabla g|^p dx.
\end{align}
Note that $B_{8r}(x) \subset B_{9r}(x_2) \cap B_{9r}(x_3)$, we obtain from~\eqref{eq:M} and~\eqref{est:Ia} the following estimates
\begin{align} \label{eq:M2r1}
\mathbf{M}^{2r}_{\alpha} (\chi_{B_{2r}(x)}|\nabla v|^p)(y) & \le (2r)^{\alpha}\|\nabla v\|^p_{L^{\infty}(B_{2r}(x))} \notag \\ 
 & \le  C_7 \left( r^{\alpha} \fint_{B_{9r}(x_2)} |\nabla u|^p dx + r^{\alpha} \fint_{B_{9r}(x_3)} |f|^{\frac{p}{p-1}} + |\nabla g|^p dx \right)  \notag \\
& \le C_8 \left(\mathbf{M}\mathbf{M}_{\alpha}(|\nabla u|^p)(x_2) + \mathbf{M}\mathbf{M}_{\alpha}(|f|^{\frac{p}{p-1}} + |\nabla g|^p)(x_3)\right) \notag \\ 
& \le C_9 \lambda,
\end{align}
and deduce from~\eqref{est:Ib} that
\begin{align} \label{eq:M2r2}
(4r)^{\alpha}\fint_{B_{4r}(x)} |\nabla u - \nabla v|^p dx & \le C_{10} (\delta^m+\gamma) r^{\alpha} \fint_{B_{9r}(x_2)} |\nabla u|^p dx \notag \\
& \qquad \qquad + C_{10}  r^{\alpha}  \fint_{B_{9r}(x_3)} |f|^{\frac{p}{p-1}} + |\nabla g|^p dx  \notag \\ 
& \le C_{10} \left(\delta^m + \gamma + \kappa \right)\lambda. 
\end{align}
It follows easily from~\eqref{eq:M2r1} that if ${\sigma} \ge \max\{4^{n},2^pC_9\}$, then
$$\left| \left\{ \mathbf{M}^{2r}_{\alpha}\left(\chi_{B_{2r}(x)}|\nabla v|^p\right)>2^{-p}{\sigma}\lambda\right\} \cap B_r(x)\right| = 0,$$
which implies from~\eqref{eq:F2} that
$$\left|V \cap B_r(x)\right| \le \left| \left\{ \mathbf{M}^{2r}_{\alpha}\left(\chi_{B_{2r}(x)}|\nabla u-\nabla v|^p\right)>2^{-p}{\sigma}\lambda\right\} \cap B_r(x)\right|.$$
Using again the bounded property of the fractional maximal function $\mathbf{M}_{\alpha}$ in Lemma~\ref{lem:Malpha}, we obtain from the above estimate and~\eqref{eq:M2r2} that
\begin{align}\label{eq:M2r3}
\left|V \cap B_r(x)\right| & \le \frac{C_{11}}{(2^{-p}{\sigma}\lambda)^{\frac{n}{n-\alpha}}} \left( \int_{B_{2r}(x)} |\nabla u - \nabla v|^p dx\right)^{\frac{n}{n-\alpha}} \notag \\ 
& \le \frac{C_{11}}{(2^{-p}{\sigma}\lambda)^{\frac{n}{n-\alpha}}} (4r)^{n} \left(  (4r)^{\alpha} \fint_{B_{4r}(x)} |\nabla u - \nabla v|^p dx\right)^{\frac{n}{n-\alpha}} \notag \\  
& \le C_{12}\left(\delta^m + \gamma + \kappa \right)^{\frac{n}{n-\alpha}} |B_r(x)|.
\end{align}
By the definition of the Muckenhoupt weight $\omega \in \mathcal{A}_{\infty}$, one may deduce~\eqref{eq:F1} from~\eqref{eq:M2r3}. That means
\begin{align*}
\omega(V \cap B_r(x)) & \le C \left(\frac{|V \cap B_r(x)|}{|B_r(x)|}\right)^\nu \omega(B_r(x)) \\
& \le C_{13}\left(\delta^m + \gamma + \kappa\right)^{\frac{n\nu}{n-\alpha}} \omega(B_r(x))  < \varepsilon \omega(B_r(x)),
\end{align*}
where $\delta$, $\kappa$ and $\gamma$ are small enough such that 
$$C_{13}\left(\delta^m + \gamma + \kappa\right)^{\frac{n\nu}{n-\alpha}} < \varepsilon.$$

Let us next consider the second case when $x$ is near the boundary $B_{8r}(x) \cap \partial \Omega \neq \emptyset$. Let $x_4 \in \partial \Omega$ such that $|x_4-x| = \mbox{dist}(x,\partial\Omega)$. We remark that
$$ B_{2r}(x) \subset B_{10r}(x_4) \subset B_{360r}(x_4) \subset B_{369r}(x) \subset B_{380r}(x_2) \cap B_{380r}(x_3).$$
Applying Theorem~\ref{theo:boundary-comp}, one can find $\tilde{v} \in L^{\infty}(B_{10r}(x_4) \cap \Omega)\cap W^{1,p}(B_{10r}(x_4) \cap \Omega)$ such that
\begin{align*}
\|\nabla \tilde{v}\|^p_{L^{\infty}(B_{10r}(x_4))} \le C_{14} \left( \fint_{B_{360r}(x_4)} |\nabla u|^p dx + \fint_{B_{360r}(x_4)} |f|^{\frac{p}{p-1}} +|\nabla g|^p dx \right),
\end{align*}
and for every $\gamma \in (0,1)$ there holds
\begin{align*}
\fint_{B_{10r}(x_4)} |\nabla u - \nabla \tilde{v}|^p dx \le C_{15} \left(\delta^m  + \gamma\right) \fint_{B_{360r}(x_4)} |\nabla u|^p dx + C_{15} \fint_{B_{360r}(x_4)} |f|^{\frac{p}{p-1}} + |\nabla g|^p dx. 
\end{align*}
As in the proof of the first case, thanks to~\eqref{eq:M} it follows from the above estimates that
\begin{align*}
(10r)^{\alpha}\|\nabla \tilde{v}\|^p_{L^{\infty}(B_{10r}(x_4))} & \le C_{16} \left(r^{\alpha} \fint_{B_{380r}(x_2)} |\nabla u|^p dx  + r^{\alpha}\fint_{B_{380r}(x_3)} |f|^{\frac{p}{p-1}} + |\nabla g|^p dx \right)\\
& \le C_{17}\left(1+\kappa\right)\lambda  \le C_{18}\lambda,
\end{align*}
and
\begin{align*}
(2r)^{\alpha}\fint_{B_{2r}(x)} |\nabla u - \nabla \tilde{v}|^p dx 
& \le C_{19} \left(\delta^m + \gamma \right) r^{\alpha}\fint_{B_{380r}(x_2)} |\nabla u|^p dx  \\
& \qquad +  C_{19} r^{\alpha}\fint_{B_{380r}(x_3)} |f|^{\frac{p}{p-1}} + |\nabla g|^p dx \\ 
& \le C_{20} \left(\delta^m + \gamma +  \kappa\right)\lambda.
\end{align*}
Therefore, for ${\sigma} \ge \max\{4^{n},2^pC_9, 2^pC_{18}\}$, we may conclude that
\begin{align*}
\left|V \cap B_r(x)\right| & \le \left|\left\{ \mathbf{M}^{2r}_{\alpha}\left(\chi_{B_{2r}(x)}|\nabla u-\nabla \tilde{v}|^p\right)>2^{-p}{\sigma}\lambda\right\}  \cap B_r(x)\right| \\
& \le \frac{C_{21}}{\left(2^{-p}{\sigma}\lambda\right)^{\frac{n}{n-\alpha}}} \left(\int_{B_{2r}(x)} |\nabla u - \nabla \tilde{v}|^pdx\right)^{\frac{n}{n-\alpha}} \\
& \le \frac{C_{21}}{\left(2^{-p}{\sigma}\lambda\right)^{\frac{n}{n-\alpha}}} (2r)^{n} \left((2r)^{\alpha}\fint_{B_{2r}(x)} |\nabla u - \nabla \tilde{v}|^pdx\right)^{\frac{n}{n-\alpha}} \\
& \le C_{22}\left(\delta^m + \gamma + \kappa \right)^{\frac{n}{n-\alpha}}|B_r(x)|.
\end{align*}
By the definition of Muckenhoupt weight $\omega \in \mathcal{A}_{\infty}$, this follows that
\begin{align*}
\omega(V \cap B_r(x)) & \le C \left(\frac{|V \cap B_r(x)|}{|B_r(x)|}\right)^\nu \omega(B_r(x)) \\
& \le C_{23}\left(\delta^m + \gamma + \kappa \right)^{\frac{n\nu}{n-\alpha}} \omega(B_r(x)).
\end{align*}
To complete the proof, all we need is to choose positive numbers $\gamma$, $\kappa$ and $\delta$ small enough such that $C_{23}\left(\delta^m + \gamma + \kappa\right)^{\frac{n\nu}{n-\alpha}} < \varepsilon$.
\end{proof}

\subsection{Riesz point-wise estimates}
\begin{proof}[Proof of Theorem~\ref{theo:regularity}]
By Theorem~\ref{theo:goodlambda}, for any $\varepsilon>0$ and $\lambda>0$, there exist some positive constants $\delta$, ${\sigma}$ and $\kappa$ such that if $\Omega$ is a $(\delta,R_0)$-Reifenberg flat domain satisfying $[A]_{R_0} \le \delta$ for some $R_0>0$ then
\begin{align}\label{eq:main-1}
\omega\left(V \right) \leq C \varepsilon \omega\left(W \right),
\end{align}
where $W = \{\mathbf{M}\mathbf{M}_{\alpha}(|\nabla u|^{p})> \lambda\}\cap \Omega$ and 
\begin{align*}
V = \left\{\mathbf{M}\mathbf{M}_{\alpha}(|\nabla u|^{p})>{\sigma}\lambda, \ \mathbf{M}_{\alpha}(|f|^{\frac{p}{p-1}}+|\nabla g|^p) \le \kappa\lambda \right\}\cap \Omega.
\end{align*}
We deduce from~\eqref{eq:main-1} that
\begin{align}\nonumber
{\omega}\left(\{{\mathbf{M}\mathbf{M}}_\alpha(|\nabla u|^p)>\sigma\lambda\}\right) &\le C\varepsilon {\omega}\left(\{{\mathbf{M}\mathbf{M}_{\alpha}}(|\nabla u|^p)>\lambda\}\cap\Omega \right)\\ \label{eq:main-2}
& \qquad \qquad + {\omega}\left(\{{\mathbf{M}}_\alpha(|f|^{\frac{p}{p-1}}+|\nabla u|^p)>\kappa\lambda\}\cap\Omega \right).
\end{align}
By the definition of the norm given in \eqref{eq:w-lorentz-a}, one has
\begin{align*}
\|{\mathbf{M}\mathbf{M}}_\alpha(|\nabla u|^p)\|^s_{L^{q,s}_{\omega}(\Omega)} = q \int_0^\infty{\lambda^s {\omega}\left(\{{\mathbf{M}\mathbf{M}}_\alpha(|\nabla u|^p)>\lambda\} \right)^{\frac{s}{q}}\frac{d\lambda}{\lambda}}.
\end{align*}
Changing the variable $\lambda$ to $\sigma\lambda$ in the integral on the right-hand side, we get that
\begin{align} \label{eq:main-3}
\|{\mathbf{M}\mathbf{M}}_\alpha(|\nabla u|^p)\|^s_{L^{q,s}_{\omega}(\Omega)} = \sigma^{s}q\int_0^\infty{\lambda^s{\omega}\left(\{{\mathbf{M}\mathbf{M}}_\alpha(|\nabla u|^p)>\sigma\lambda\} \right)^{\frac{s}{q}}\frac{d\lambda}{\lambda}}.
\end{align}
Thanks to~\eqref{eq:main-2}, it follows from~\eqref{eq:main-3} that
\begin{align*}
\|{\mathbf{M}\mathbf{M}}_\alpha(|\nabla u|^p)\|^s_{L^{q,s}_{\omega}(\Omega)} &\le \sigma^{s}(2C\varepsilon)^{\frac{s}{q}}q\int_0^\infty{\lambda^s{\omega}\left(\{{\mathbf{M}\mathbf{M}}_\alpha(|\nabla u|^p)>\lambda\}\cap\Omega \right)^{\frac{s}{q}}\frac{d\lambda}{\lambda}}\\
&~~~+ \sigma^{s}2^{\frac{s}{q}} q \int_0^\infty{\lambda^s{\omega}\left(\{{\mathbf{M}}_\alpha(|f|^{\frac{p}{p-1}}+|\nabla \sigma|^p)>\kappa\lambda\}\cap\Omega \right)^{\frac{s}{q}}\frac{d\lambda}{\lambda}}.
\end{align*}
Making the change of variables again in the second integral on right-hand side of above estimate, it is straightforward to obtain
\begin{align*}
\|{\mathbf{M}\mathbf{M}}_\alpha(|\nabla u|^p)\|^s_{L^{q,s}_{\omega}(\Omega)} &\le \sigma^{s}(2C\varepsilon)^{\frac{s}{q}}\|{\mathbf{M}\mathbf{M}}_\alpha\left(|\nabla u|^p \right)\|^s_{L^{q,s}_{\omega}(\Omega)}\\ & \qquad \qquad + \sigma^{s}2^{\frac{s}{q}} \kappa^{-s} \|{\mathbf{M}}_\alpha(|f|^{\frac{p}{p-1}}+|\nabla \sigma|^p)\|^s_{L^{q,s}_{\omega}(\Omega)},
\end{align*}
and achieve the proof of the desired estimate by taking $\sigma^{s}(2C\varepsilon)^{\frac{s}{q}} \le \frac{1}{2}$.
\end{proof}

To complete the last study of this section, we also prove Theorem \ref{theo:I-alpha} which relates a point-wise gradient bound for solutions to problem \eqref{eq:diveq} in terms of the Riesz potentials. 

\begin{proof}[Proof of Theorem~\ref{theo:I-alpha}]
Applying Theorem~\ref{theo:regularity}, we obtain at once that for any $\alpha \in [0,n)$, $0<q<\infty$ and $\omega \in \mathcal{A}_\infty$, there exists a positive constant $\delta$ such that if $\Omega$ is a $(\delta,R_0)$-Reifenberg flat domain satisfying $[A]_{R_0} \le \delta$ for some $R_0>0$, then
\begin{align*}
\|\mathbf{M}_{\alpha}(|\nabla u|^p)\|_{L^{q}_{\omega}(\Omega)} \le C \|\mathbf{M}_{\alpha}(|f|^{\frac{p}{p-1}} + |\nabla g|^p)\|_{L^{q}_{\omega}(\Omega)}, 
\end{align*}
which is equivalent to 
\begin{align}\label{pw-1}
\left(\int_{\Omega} |\mathbf{M}_{\alpha}(|\nabla u|^p)|^q \omega(x) dx\right)^{\frac{1}{q}} \le C \left(\int_{\Omega} |\mathbf{M}_{\alpha}(|f|^{\frac{p}{p-1}} + |\nabla g|^p)|^q \omega(x) dx\right)^{\frac{1}{q}}. 
\end{align}
For any $z \in \mathbb{R}^n$ and $\varepsilon>0$ small enough, let us set $h = \chi_{B_{\varepsilon}(z)}$ and $\tilde{\omega} = \mathbf{I}_{\beta}h$. We will show that $\tilde{\omega} \in \mathcal{A}_1$. Indeed, we now consider $\omega_0(x) = |x|^{1-n}$, $ x \in \mathbb{R}^n$. It is easily seen that $\omega_0 \in \mathcal{A}_1$. In other words, there exists a constant $C_0>0$ such that  
\begin{align}\label{pw-2}
\mathbf{M}(\omega_0)(x) \le C_0 \omega_0 (x), \quad \forall x \in \mathbb{R}^n.
\end{align}
Thanks to Fubini's Theorem, one may conclude from~\eqref{pw-2} that 
\begin{align*}
\mathbf{M}(\tilde{\omega})(x) \le C_0 \tilde{\omega} (x), \quad \forall x \in \mathbb{R}^n,
\end{align*}
which implies that $\tilde{\omega} \in \mathcal{A}_1$. Therefore, one can apply~\eqref{pw-1} with $\omega = \tilde{\omega}$ to arrive
\begin{align*}
& \int_{\mathbb{R}^n} \chi_{\Omega}(x) |\mathbf{M}_{\alpha}(|\nabla u|^p)(x)|^q  \int_{\mathbb{R}^n}  \frac{\chi_{B_{\varepsilon}(z)}(y)}{|y-x|^{n-\beta}} dy dx \\ & \hspace{2cm}  \le C \int_{\mathbb{R}^n} \chi_{\Omega}(x) |\mathbf{M}_{\alpha}(|f|^{\frac{p}{p-1}} + |\nabla g|^p)(x)|^q  \int_{\mathbb{R}^n}  \frac{\chi_{B_{\varepsilon}(z)}(y)}{|y-x|^{n-\beta}} dy dx.
\end{align*} 
By Fubini's theorem, it gives
\begin{align*}
& \int_{\mathbb{R}^n} \chi_{B_{\varepsilon}(z)}(y) \int_{\mathbb{R}^n} \frac{\chi_{\Omega}(x)|\mathbf{M}_{\alpha}(|\nabla u|^p)(x)|^q}{|y-x|^{n-\beta}} dx dy \\ & \hspace{2cm} \le C \int_{\mathbb{R}^n} \chi_{B_{\varepsilon}(z)}(y) \int_{\mathbb{R}^n}\frac{\chi_{\Omega}(x)|\mathbf{M}_{\alpha}(|f|^{\frac{p}{p-1}} + |\nabla g|^p)(x)|^q}{|y-x|^{n-\beta}} dx dy.
\end{align*}
And it follows that
\begin{align}\label{est:Ifg}
\fint_{B_{\varepsilon}(z)} \mathbf{I}_{\beta}\left(\chi_{\Omega}|\mathbf{M}_{\alpha}(|\nabla u|^p)|^q \right)(y)dy \le C \fint_{B_{\varepsilon}(z)} \mathbf{I}_{\beta}\left( \chi_{\Omega}|\mathbf{M}_{\alpha}(|f|^{\frac{p}{p-1}} + |\nabla g|^p)|^q \right)(y)dy.
\end{align}
By passing $\varepsilon$ to $0$ in \eqref{est:Ifg}, one concludes the following point-wise inequality:
\begin{align*}
\mathbf{I}_{\beta}\left(\chi_{\Omega}|\mathbf{M}_{\alpha}(|\nabla u|^p)|^q \right)(x) \le C \mathbf{I}_{\beta}\left( \chi_{\Omega}|\mathbf{M}_{\alpha}(|f|^{\frac{p}{p-1}} + |\nabla g|^p)|^q \right)(x), 
\end{align*}
holds for almost everywhere $x \in \mathbb{R}^n$. The proof is then complete.
\end{proof}

\section{Applications}\label{sec:app}
In this section, we apply the point-wise estimate in Theorem~\ref{theo:I-alpha} to study the solvability of the generalized equation~\eqref{eq:existI}:
\begin{align*}
\begin{cases} -\mathrm{div}(A(x,\nabla u)) & = \ \mathbf{I}_{\beta}(|\nabla u|^p)^q + \mathrm{div}(f), \ \mbox{ in } \ \Omega, \\ \hspace{1.5cm} u & = \  g, \ \mbox{ on } \partial \Omega,\end{cases}
\end{align*}
where $\beta \in (0,n)$, $p>1$, $q>0$. In addition, problem \eqref{eq:existI} is set under suitable assumptions on the boundary of domain $\Omega$ and the coefficients $A$, the same hypotheses of known data $f$, $g$ in Theorem~\ref{theo:I-alpha}. In particular, in this section we will show that this equation admits at least one solution under an additional Riesz capacity condition on the data. Moreover, we also show that a type of this Riesz capacity condition is necessary for the existence result. Our key point is based on some comparison estimates on Riesz and Wolff potentials which are firstly discussed in the next subsection.
\subsection{Comparisons on Riesz and Wolff potentials}
Let us first recall the following lemma which links a condition on the Wolff potential of a measure with a Riesz capacity assumption in the whole space.  Here, we send the interested reader to~\cite{Phuc2008} for the proof of this lemma.
\begin{proposition}\label{prop:Hung-Veron}
Let $1 < \beta_2 < \frac{n}{\beta_1}$, $s > \beta_2 -1$ and $\nu  \in \mathcal{M}^+(\mathbb{R}^n)$. Two following statements are equivalent:
\begin{itemize}
\item[(i)] The inequality
\begin{align*}
\nu (K) \le c \ \mathrm{Cap}_{\mathbf{I}_{\beta_1\beta_2},\frac{s}{s-\beta_2+1}}(K)
\end{align*}
holds for any compact set $K \subset \mathbb{R}^n$, for a constant $c$.
\item[(ii)] The inequality
\begin{align*}
\int_{\mathbb{R}^n} \left(\mathbf{W}_{\beta_1,\beta_2}(\chi_{B_r(x)}\nu )(y)\right)^s dy \le c \nu (B_r(x))
\end{align*}
holds for any ball $B_r(x) \subset \mathbb{R}^n$.
\end{itemize}
\end{proposition}

The next lemma is directly a consequence of~\cite[Lemma 2.1]{BNV2018}. The detail proof can be also found in~\cite{BNV2018}.
\begin{lemma}\label{lem:Hung-Veron}
Let $k, m \in \mathbb{R}^+$ and $s \in \mathbb{R}$. Assume that $H: \mathbb{R}^+ \to \mathbb{R}^+$ is a non-decreasing function. There exists a positive constant $C=C(k,m,s)$ such that
\begin{align*}
\int_0^{\infty} \varrho^{k} \left(\int_{\varrho}^{\infty} \frac{H(r)}{r^{s}} \frac{dr}{r}\right)^{m} \frac{d\varrho}{\varrho} \le C \int_0^{\infty} \varrho^{k} \left( \frac{H(\varrho)}{\varrho^{s}}\right)^{m} \frac{d\varrho}{\varrho}.
\end{align*}
\end{lemma}

Using Lemma~\ref{lem:Hung-Veron} and Proposition~\ref{prop:Hung-Veron}, we can perform some comparison estimates between Riesz and Wolff potentials.

\begin{lemma}\label{lem:2I} Let $\beta_1, \beta_2 \in (0,n)$. 
If $q>0$ and $\nu  \in \mathcal{M}^+(\mathbb{R}^n)$ then there exists $C_1$ depending on $\beta_1, \beta_2, n, q$ such that
\begin{align}\label{eq:lem-2Ia}
\mathbf{I}_{{\beta_2}}\left(\mathbf{I}_{{\beta_1}}(\nu )^q\right)(x) \ge C_1 \mathbf{W}_{\frac{q{\beta_1} + {\beta_2}}{q+1}, \frac{1}{q}+1}(\nu )(x), \quad \mbox{ in } \mathbb{R}^n.
\end{align}
If $0<q<\frac{n}{n-{\beta_1}}$ and $\nu  \in \mathcal{M}^+(\mathbb{R}^n)$ then there exists $C_2$ depending on $\beta_1, \beta_2, n, q$ such that
\begin{align}\label{eq:lem-2I}
\mathbf{I}_{{\beta_2}}\left(\mathbf{I}_{{\beta_1}}(\nu )^q\right)(x) \le C_2 \mathbf{W}_{\frac{q{\beta_1} + {\beta_2}}{q+1}, \frac{1}{q}+1}(\nu )(x), \quad \mbox{ in } \mathbb{R}^n.
\end{align}
\end{lemma}
\begin{proof} We remark that the Riesz potential $\mathbf{I}_{{\beta_1}}$ defined by~\eqref{def:Riesz} can be also rewritten as the following form
\begin{align*}
\mathbf{I}_{{\beta_1}}(\nu )(x) = \int_0^{\infty} \frac{\nu (B_r(x))}{r^{n-{\beta_1}}} \frac{dr}{r}.
\end{align*}
For every $x \in \mathbb{R}^n$, with notice that $B_{\varrho}(x) \subset B_{2\varrho}(y)$ for all $y \in B_{\varrho}(x)$, one gets that
\begin{align*}
\mathbf{I}_{{\beta_2}}\left(\mathbf{I}_{{\beta_1}}(\nu )^q\right)(x) & = \int_0^{\infty} \frac{1}{\varrho^{n-{\beta_2}}} \int_{B_{\varrho}(x)} \left( \int_0^{\infty} \frac{\nu (B_r(y))}{r^{n-{\beta_1}}} \frac{dr}{r} \right)^qdy \frac{d\varrho}{\varrho} \\
& \ge C \int_0^{\infty} \frac{1}{\varrho^{n-{\beta_2}}} \int_{B_{\varrho}(x)} \left( \int_0^{\infty} \frac{\nu (B_{2\varrho}(y))}{\varrho^{n-{\beta_1}}} dy \right)^q \frac{d\varrho}{\varrho} \\
& \ge C_1 \int_0^{\infty} \varrho^{{\beta_2}}  \left(  \frac{\nu (B_{\varrho}(x))}{\varrho^{n-{\beta_1}}} \right)^q \frac{d\varrho}{\varrho} \\
& = C_1 \int_0^{\infty}  \left(  \frac{\nu (B_{\varrho}(x))}{\varrho^{n-{\beta_1} -\frac{{\beta_2}}{q}}} \right)^q \frac{d\varrho}{\varrho} \\
& = C_1 \mathbf{W}_{\frac{q{\beta_1 + {\beta_2}}}{q+1},\frac{1}{q}+1}(\nu )(x),
\end{align*}
which is exactly~\eqref{eq:lem-2Ia}. On the other hand, for $0<q<\frac{n}{n-{\beta_1}}$, we recall the following estimate on Riesz potential
\begin{align*}
\|\mathbf{I}_{{\beta_1}}(\tilde{\nu })\|_{L^{\frac{n}{n-{\beta_1}},\infty}} \le C \tilde{\nu }({\mathbb{R}^n}), \quad \forall \tilde{\nu } \in \mathcal{M}^+_b(\mathbb{R}^n),
\end{align*}
which guarantees that 
\begin{align*}
\int_{B_r(x)}\mathbf{I}_{{\beta_1}}(\tilde{\nu })^q dy \le C r^n \left(\frac{\tilde{\nu }(\mathbb{R}^n)}{r^{n-{\beta_1}}}\right)^q, \quad \forall x \in \mathbb{R}^n, \ \forall r>0.
\end{align*}
Applying this inequality for $\tilde{\nu } = \chi_{B_{2r}(x)} \nu $, one has
\begin{align*}
\int_{B_r(x)}\mathbf{I}_{{\beta_1}}(\chi_{B_{2r}(x)} \nu )^q dy \le C r^n \left(\frac{\nu ({B_{2r}(x)})}{r^{n-{\beta_1}}}\right)^q, \quad \forall x \in \mathbb{R}^n, \ \forall r>0.
\end{align*}
Basing on this fact and notice that for all $r \ge \varrho>0$, since $B_r(y) \subset B_{2r}(x)$ for any $y \in B_{\varrho}(x)$, we may estimate as below
\begin{align*}
\mathbf{I}_{{\beta_2}}\left(\mathbf{I}_{{\beta_1}}(\nu )^q\right)(x) & \le C\int_0^{\infty} \frac{1}{\varrho^{n-{\beta_2}}} \int_{B_{\varrho}(x)} \left( \int_{\varrho}^{\infty} \frac{\nu (B_r(y))}{r^{n-{\beta_1}}} \frac{dr}{r} \right)^qdy \frac{d\varrho}{\varrho} \\
& \le C \int_0^{\infty} \frac{1}{\varrho^{n-{\beta_2}}} \int_{B_{\varrho}(x)} \left( \int_{\varrho}^{\infty} \frac{\nu (B_{2r}(x))}{r^{n-{\beta_1}}} \frac{dr}{r} \right)^q dy \frac{d\varrho}{\varrho} \\
& \le C \int_0^{\infty} {\varrho^{{\beta_2}}}  \left( \int_{\varrho}^{\infty} \frac{\nu (B_{2r}(x))}{r^{n-{\beta_1}}} \frac{dr}{r} \right)^q \frac{d\varrho}{\varrho}.
\end{align*}
Thanks to Lemma~\ref{lem:Hung-Veron}, we obtain from above inequality that
\begin{align*}
\mathbf{I}_{{\beta_2}}\left(\mathbf{I}_{{\beta_1}}(\nu )^q\right)(x) & \le C_2 \int_0^{\infty} {\varrho^{{\beta_2}}}  \left( \frac{\nu (B_{\varrho}(x))}{\varrho^{n-{\beta_1}}}\right)^q \frac{d\varrho}{\varrho},
\end{align*}
which leads to~\eqref{eq:lem-2I}. The proof is complete.
\end{proof}

\begin{lemma}\label{lem:cond-varphi} 
Let ${\beta_1}, \, {\beta_2}, \, \beta_3 \in (0,n)$, $0<q<\frac{n}{n-{\beta_1}}$, $qs >1$ and $\nu  \in \mathcal{M}^+(\mathbb{R}^n)$. Assume that the following inequality
\begin{align}\label{eq:cond-varphi-1}
\nu (K) \le \theta \ \mathrm{Cap}_{\mathbf{I}_{\beta_1 + \frac{\beta_2}{q}},\frac{qs}{qs-1}}(K),
\end{align}
holds for any compact set $K \subset \mathbb{R}^n$ and for a constant $\theta$. There holds
\begin{align}\label{re-lem:2I}
\mathbf{I}_{{\beta_3}}\left(\mathbf{I}_{{\beta_2}}\left(\mathbf{I}_{{\beta_1}}(\nu )^q\right)\right)^s(x) \le C \theta \ \mathbf{I}_{{\beta_3}}(\nu )(x), \quad \mbox{ in } \ \mathbb{R}^n.
\end{align}
\end{lemma}
\begin{proof}
By Proposition~\ref{prop:Hung-Veron}, the fact that the inequality~\eqref{eq:cond-varphi-1} holds for any compact set $K \subset \mathbb{R}^n$ is equivalent to
\begin{align}\label{eq:lem-2I-a}
\int_{\mathbb{R}^n} \left(\mathbf{W}_{\frac{q\beta_1 +\beta_2}{q+1},\frac{1}{q}+1}(\chi_{B_r(x)}\nu )(y)\right)^s dy \le \theta \nu (B_r(x))
\end{align}
holds for any ball $B_r(x) \subset \mathbb{R}^n$.
Applying Lemma~\ref{lem:2I}, there holds
\begin{align}\label{eq:lem-2I-b}
\mathbf{I}_{{\beta_2}}\left(\mathbf{I}_{{\beta_1}}(\nu )^q\right)(x) \le C_1\mathbf{W}_{\frac{q\beta_1 +\beta_2}{q+1},\frac{1}{q}+1}(\nu )(x), \quad \mbox{ in } \ \mathbb{R}^n.
\end{align}
It follows from~\eqref{eq:lem-2I-a} and~\eqref{eq:lem-2I-b} that
\begin{align*}
\mathbf{I}_{{\beta_3}}\left(\mathbf{I}_{{\beta_2}}\left(\mathbf{I}_{{\beta_1}}(\nu )^q\right)\right)^s(x) & = \int_0^{\infty} \frac{1}{\varrho^{n-\beta_3}}  \left(\int_{B_{\varrho}(x)} \left( \mathbf{I}_{{\beta_2}}\left(\mathbf{I}_{{\beta_1}}(\nu )^q\right)(y) \right)^s dy \right) \frac{d\varrho}{\varrho} \\
& \le  \int_0^{\infty} \frac{1}{\varrho^{n-\beta_3}}  \left(\int_{B_{\varrho}(x)} \left( C_1\mathbf{W}_{\frac{q\beta_1 +\beta_2}{q+1},\frac{1}{q}+1}(\nu )(y) \right)^s dy \right) \frac{d\varrho}{\varrho}\\
& \le C_1^s \theta\int_0^{\infty} \frac{\nu (B_{\varrho}(x)) }{\varrho^{n-\beta_3}}  \frac{d\varrho}{\varrho},
\end{align*}
which leads to~\eqref{re-lem:2I}. The proof is complete.
\end{proof}

\subsection{Existence result}
We now prove Theorem~\ref{theo:existence_Ialpha} which presents a sufficient condition for the existence of a solution to equation~\eqref{eq:existI}. We start with the following lemma which can be obtained from Theorem~\ref{theo:I-alpha}.
\begin{lemma}\label{lem:app-1}
Let $p>1$,  $f \in L^{\frac{p}{p-1}}(\Omega; \mathbb{R}^n)$, $g \in W^{1,p}(\Omega;\mathbb{R})$ and $\mathbf{I}_1(\eta ) \in  L^{\frac{p}{p-1}}(\Omega; \mathbb{R})$.  Let $u$ be a weak solution to the following equation
\begin{align}\label{eq-lem:app-1}
\begin{cases} -\mathrm{div}(A(x,\nabla u)) & = \ \eta  + \mathrm{div}(f), \ \mbox{ in } \ \Omega, \\ \hspace{1.5cm} u & = \  g, \ \mbox{ on } \partial \Omega.\end{cases}
\end{align}
There exists a constant $\delta>0$ such that if $\Omega$ is a $(\delta,R_0)$-Reifenberg flat domain satisfying $[A]_{R_0} \le \delta$ for some $R_0>0$, then for $\beta \in (0,n)$
\begin{align*}
\mathbf{I}_{\beta} (|\nabla u|^p)(x) \le C^* \left(\mathbf{I}_{\beta} (\mathbf{I}_1(\eta )^{\frac{p}{p-1}}) + \mathbf{I}_{\beta}(|f|^{\frac{p}{p-1}} + |\nabla g|^p)\right)(x),
\end{align*}
for almost everywhere $x \in \mathbb{R}^n$ and for a constant $C^*>0$.
\end{lemma}
\begin{proof}
Let $B_R:=B_R(0) \supset \Omega$ and $u_0$ be the unique solution to the following equation
\begin{align*}
\begin{cases} - \Delta u_0 & = \ \eta  , \ \mbox{ in } \ B_R, \\ \hspace{0.5cm} u_0 & = \  0, \ \mbox{ on } \partial B_R.\end{cases}
\end{align*}
It is well known that the fundamental solution to this equation satisfies 
\begin{align}\label{est:fund}
|\nabla u_0(x)| \le |\nabla_x G(\eta )(x)| \le C_1\mathbf{I}_1(\eta )(x),
\end{align}
where $G$ denotes the Green kernel. Applying Theorem~\ref{theo:I-alpha} to equation~\eqref{eq-lem:app-1}, we can find $\delta>0$ such that if $\Omega$ is a $(\delta,R_0)$-Reifenberg flat domain satisfying $[A]_{R_0} \le \delta$ for some $R_0>0$, then for $\beta \in (0,n)$
\begin{align*}
\mathbf{I}_{\beta}(|\nabla u|^p)  \le C_2 \left(\mathbf{I}_{\beta}(|\nabla u_0|^{\frac{p}{p-1}}) + \mathbf{I}_{\beta}(|f|^{\frac{p}{p-1}} + |\nabla g|^p) \right).
\end{align*}
It deduces from~\eqref{est:fund} to
\begin{align*}
\mathbf{I}_{\beta}(|\nabla u|^p)  \le  C^* \left(\mathbf{I}_{\beta} (\mathbf{I}_1(\eta )^{\frac{p}{p-1}}) + \mathbf{I}_{\beta}(|f|^{\frac{p}{p-1}} + |\nabla g|^p)\right),
\end{align*}
which finishes the proof.
\end{proof}

\medskip

\begin{proof}[Proof of Theorem~\ref{theo:existence_Ialpha}]
Let us first introduce a set $\mathcal{S}$ defined by
\begin{align}\label{S-def}
\mathcal{S} = \left\{v \in W^{1,p}(\Omega): \ \ \mathbf{I}_{\alpha}(|\nabla v|^p) \le \Lambda \mathbf{I}_{\alpha} (|\mathcal{F}|^p) \ \mbox{ a.e. in } \ \mathbb{R}^n\right\},
\end{align}
where the positive constant $\Lambda$ will be determined later. For every $v \in \mathcal{S}$, we define $T(v) := u$ as the unique solution to the following equation
\begin{align}\label{eq:app-2}
\begin{cases} -\mathrm{div}(A(x,\nabla u)) & = \ \mathbf{I}_{\beta}(|\nabla v|^p)^q + \mathrm{div}(f), \ \mbox{ in } \ \Omega, \\ \hspace{1.5cm} u & = \  g, \ \mbox{ on } \partial \Omega.\end{cases}
\end{align}
We next to show that one can find $\Lambda>0$ such that the mapping  $T:  \mathcal{S} \to \mathcal{S}$, $v \mapsto T(v) = u$ defined by~\eqref{eq:app-2} is well-defined.  In other words, we need to prove that $u = T(v) \in \mathcal{S}$ for all $v \in \mathcal{S}$. Indeed, thanks to Lemma~\ref{lem:app-1}, one obtains the following estimate
\begin{align}\label{est:app-1}
\mathbf{I}_{\alpha}(|\nabla u|^p) \le C^* \left(\mathbf{I}_{\alpha}(\mathbf{I}_1(\eta)^{\frac{p}{p-1}}) + \mathbf{I}_{\alpha}(|\mathcal{F}|^p)  \right), \quad \mbox{ a.e. in } \ \mathbb{R}^n,
\end{align}
where $\eta = (\mathbf{I}_{\beta}(|\nabla v|^p))^q \le \Lambda^q \mathbf{I}_{\beta}(|\mathcal{F}|^p)^q$ for $v \in \mathcal{S}$. It follows from~\eqref{est:app-1} that
\begin{align}\label{est:app-2}
\mathbf{I}_{\alpha}(|\nabla u|^p) \le C^* \left(\Lambda^{\frac{pq}{p-1}}\mathbf{I}_{\alpha}(\mathbf{I}_1(\mathbf{I}_{\beta}(|\mathcal{F}|^p)^q)^{\frac{p}{p-1}}) + \mathbf{I}_{\alpha}(|\mathcal{F}|^p)  \right), \quad \mbox{ a.e. in } \ \mathbb{R}^n.
\end{align}
Applying Lemma~\ref{lem:cond-varphi} under condition~\eqref{eq:ReiszCapa_cond} with $d\nu = |\mathcal{F}|^pdx$, $\beta_1 = \beta$, $\beta_2 = 1$, $\beta_3 = \alpha$ and $s = \frac{p}{p-1}$, there holds
\begin{align*}
\mathbf{I}_{\alpha}\left(\mathbf{I}_{1}\left(\mathbf{I}_{\beta}(|\mathcal{F}|^p)^q\right)^{\frac{p}{p-1}}\right)  \le C\varepsilon \mathbf{I}_{\alpha}(|\mathcal{F}|^p), \quad \mbox{ a.e. in } \ \mathbb{R}^n.
\end{align*}
Combining this inequality to~\eqref{est:app-2}, we can conclude that
\begin{align*}
\mathbf{I}_{\alpha}(|\nabla u|^p) \le C^* (C\varepsilon \Lambda^{\frac{pq}{p-1}} + 1) \mathbf{I}_{\alpha}(|\mathcal{F}|^p) \le \Lambda \mathbf{I}_{\alpha}(|\mathcal{F}|^p), \quad \mbox{ a.e. in } \ \mathbb{R}^n,
\end{align*}
which yields that $u \in \mathcal{S}$. We remark that in the last inequality, we may easily choose $\Lambda>0$ such that $C^* (C\varepsilon \Lambda^{\frac{pq}{p-1}} + 1) \le \Lambda$ for some $\varepsilon$ small enough. 

On the other hand, it is easy to check that the set $\mathcal{S}$ is convex, closed, the mapping $T$ is continuous and $T(S)$ is precompact under the strong topology of $W^{1,p}(\Omega)$. By the Schauder fixed point theorem,  the mapping $T$ admits at least one fixed point in $\mathcal{S}$. Finally, the estimate~\eqref{eq:theo-D} is obtained by the definition of $\mathcal{S}$ in~\eqref{S-def}. It completes the proof.
\end{proof} 
\medskip

We now give a proof of Theorem~\ref{theoE}. We emphasize that the equation~\eqref{eq:existI-b} considered in this theorem is simpler than~\eqref{eq:existI} just for simplicity of the computation.

\medskip

\begin{proof}[Proof of Theorem~\ref{theoE}]
It is well-known that if $\Omega$ is a $(\delta,R_0)$-Reifenberg flat domain satisfying $[A]_{R_0} \le \delta$ for some positive constants $\delta, R_0$ and $u$ is a renormalized solution to the following equation
\begin{align*}
\begin{cases} -\mathrm{div}(A(x,\nabla v)) & = \ \eta, \ \mbox{ in } \ \Omega, \\ \hspace{1.5cm} u & = \  0, \ \mbox{ on } \partial \Omega,\end{cases}
\end{align*}
then there exists a constant $C$ such that
\begin{align*}
\int_{\mathbb{R}^n} \mathbf{I}_1(\eta)^{\frac{p}{p-1}} d\omega \le C \int_{\mathbb{R}^n} |\nabla v|^p d\omega,
\end{align*}
for all $\omega \in \mathcal{A}_{\frac{p}{p-1}} \subset \mathcal{A}_1$.
Applying this fact for the solution $u$ to equation~\eqref{eq:existI-b}, we obtain that
\begin{align*}
\mathbf{I}_{\beta}(\mathbf{I}_1(\mathbf{I}_\beta(|\nabla u|^p)^q + \mu)^{\frac{p}{p-1}}) \le C \mathbf{I}_{\beta}(|\nabla u|^{p}) \quad \mbox{ a.e. in } \ \mathbb{R}^n,
\end{align*}
which follows that
$\mathbf{I}_{\beta}(\mathbf{I}_1(\nu)^{\frac{p}{p-1}})^q \le C \nu \ \mbox{ a.e. in } \ \mathbb{R}^n,$
for $\nu = \mathbf{I}_\beta(|\nabla u|^p)^q + \mu$. Applying~\eqref{eq:lem-2Ia} in Lemma~\ref{lem:2I}, one gets
\begin{align*}
\left(\mathbf{W}_{\frac{p+(p-1)\beta}{2p-1},\frac{2p-1}{p}} (\nu)\right)^q \le C \nu \quad \mbox{ a.e. in } \ \mathbb{R}^n.
\end{align*}
It implies that the following inequality holds for any ball $B_r(x) \subset \mathbb{R}^n$:
\begin{align*}
\int_{\mathbb{R}^n}\left(\mathbf{W}_{\frac{p+(p-1)\beta}{2p-1},\frac{2p-1}{p}} (\chi_{B_r(x)}\nu)(y)\right)^q dy \le C \nu(B_r(x)), 
\end{align*}
which yields that $\nu(K) \le C \mathrm{Cap}_{\mathbf{I}_{\beta + 1 - \frac{\beta}{q}},\frac{pq}{pq-p+1}}(K),$
for any compact subset $K \subset \mathbb{R}^n$. This implies to~\eqref{eq:ReiszCapa_cond-b} with notice that $\mu(K) \le \nu(K)$.
\end{proof}

\subsection*{Acknowledgment}
The author T.-N. Nguyen was supported by Ho Chi Minh City University of Education.

\end{document}